\documentclass{article}
\usepackage{amsmath,amssymb,amsthm, amsfonts, mathrsfs}
\usepackage{calligra,indentfirst,epsfig, enumitem, lettrine}
\calligra
\newcommand{\Fqn}{\mathbb{F}_q^{n}}
\newcommand{\Fq}{\mathbb{F}_q}

\newtheorem{thm}{Theorem}[section]
\newtheorem{lem}{Lemma}[section]
\newtheorem{cor}{Corollary}[section]
\newtheorem{prop}{Proposition}[section]
\theoremstyle{remark}
\newtheorem{rem}{Remark}[section]
\newtheorem{ex}{Example}[section]
\theoremstyle{definition}

\allowdisplaybreaks
\numberwithin{equation}{section}

\begin{document}

\title{Multi-twisted  codes over finite fields and their dual codes}
\author{Anuradha Sharma{\footnote{Corresponding Author, Email address: anuradha@iiitd.ac.in}, Varsha Chauhan\thanks{Research support by UGC, India, is gratefully acknowledged.} ~and Harshdeep Singh
}\\
{\it Department of  Mathematics, IIIT-Delhi}\\{\it New Delhi 110020, India}}
\date{}
\maketitle
\begin{abstract}
Let $\mathbb{F}_{q}$ denote the finite field of order $q,$ let $m_1,m_2,\cdots,m_{\ell}$ be positive integers satisfying $\gcd(m_i,q)=1$ for $1 \leq i \leq \ell,$ and let $n=m_1+m_2+\cdots+m_{\ell}.$ Let $\Lambda=(\lambda_1,\lambda_2,\cdots,\lambda_{\ell})$ be fixed, where $\lambda_1,\lambda_2,\cdots,\lambda_{\ell}$ are non-zero elements of $\mathbb{F}_{q}.$  In this paper, we study the algebraic structure of $\Lambda$-multi-twisted codes of length $n$ over $\mathbb{F}_{q}$ and their dual codes with respect to the  standard  inner product on $\mathbb{F}_{q}^n.$  We   provide necessary and sufficient conditions for the existence of a self-dual $\Lambda$-multi-twisted  code of length $n$ over $\mathbb{F}_{q},$ and  obtain enumeration formulae for all self-dual and self-orthogonal  $\Lambda$-multi-twisted codes of length $n$ over $\mathbb{F}_{q}.$  We also derive some sufficient conditions under which a $\Lambda$-multi-twisted code is LCD.   We  determine the parity-check polynomial of all $\Lambda$-multi-twisted codes of length $n$ over $\mathbb{F}_{q}$ and obtain a BCH type bound on their  minimum Hamming distances. We also determine generating sets  of dual codes of some $\Lambda$-multi-twisted codes of length $n$ over $\mathbb{F}_{q}$ from the generating sets of the codes.  Besides this, we provide a trace description for all $\Lambda$-multi-twisted codes of length $n$ over $\mathbb{F}_{q}$ by viewing these codes as direct sums of certain concatenated codes, which leads to a method to construct these codes.   We also obtain a lower bound on their minimum Hamming distances using their multilevel concatenated structure.
\end{abstract}

\section{Introduction}\label{intro}
\lettrine[lines=2]{\textbf{P}}{range} \cite{prange}
 first introduced and studied cyclic codes over finite fields, which form an important class of linear codes and can be effectively encoded and decoded using shift registers. Later, Townsend and Weldon \cite{weldon1} introduced and studied quasi-cyclic (QC) codes over finite fields, which are generalizations of cyclic codes. Kasami \cite{kasami} and Weldon \cite{weldon} further showed that these codes are asymptotically good due to their abundant population. Gulliver \cite{gulliver}  produced many record breaker QC codes in short lengths. Solomon and Tilborg \cite{tilborg}  established a link between these codes and
  convolutional codes.  Using this, they deduced many interesting properties of linear codes, which have  applications in coding theory and modulation. Ling and Sol\'{e} \cite{Ling} viewed QC codes over a finite field as  linear codes over a certain auxiliary ring and further studied their dual codes with respect to the standard inner product. In the same work, they also provided a trace description for all QC codes. They also explored the existence of some self-dual QC codes and enumerated this class of codes in some special cases. Siap and Kulhan \cite{Siap} introduced 
 generalized quasi-cyclic (GQC) codes over finite fields, which are generalizations of QC codes. They further studied  $1$-generator GQC codes and obtained a BCH type bound on their minimum Hamming distances.  Esmaeili and Yari \cite{Yari1}  further decomposed  GQC codes into linear codes using  Chinese Remainder Theorem and the results derived in  Ling and Sol\'{e} \cite{Ling3}. They also obtained an improved lower  bound on their  minimum Hamming distances.   G\"{u}neri et al. \cite{cem} decomposed GQC codes as direct sums of concatenated codes, which leads to a trace formula and a minimum distance bound for GQC codes.   Jia \cite{yan} decomposed quasi-twisted (QT) codes and their dual codes  over finite fields to a direct sum of  linear codes over rings, and  provided a method  to construct quasi-twisted codes by using generalized discrete Fourier transform. Saleh and Esmaeili \cite{saleh} gave some sufficient conditions under which a quasi-twisted code is LCD. In a recent work, Aydin and  Halilovic \cite{aydin} introduced  multi-twisted (MT) codes as generalization of quasi-twisted codes. They studied basic properties  of 1-generator multi-twisted codes and provided a lower bound on their minimum Hamming distances. The family of multi-twisted  codes is much broader as compared to quasi-twisted and constacyclic codes. 

Throughout this paper, let $\mathbb{F}_{q}$ denote the finite field of order $q,$ and let $\Lambda=(\lambda_1,\lambda_2,\cdots,\lambda_{\ell})$ be fixed, where $\lambda_1,\lambda_2,\cdots, \lambda_{\ell}$ are non-zero elements of $\mathbb{F}_{q}.$ Let $n=m_1+m_2+\cdots+m_{\ell},$ where $m_1,m_2,\cdots, m_{\ell}$ are positive integers coprime to $q.$  The main aim of this paper is to  study the algebraic structure of  $\Lambda$-multi-twisted codes of length $n$ over $\mathbb{F}_{q}$ and their dual codes with respect to the standard inner product on $\mathbb{F}_{q}^n.$  Enumeration formulae for  their two interesting subclasses, {\it viz.} self-dual and self-orthogonal  codes, are also obtained.   These enumeration formulae are useful in the determination of complete lists of inequivalent self-dual and self-orthogonal  $\Lambda$-multi-twisted codes \cite[Section 9.6]{huffbook}.   
Some sufficient conditions under which a $\Lambda$-multi-twisted code is LCD are also derived. Generating sets of dual codes of some $\Lambda$-multi-twisted codes are expressed in terms of   generating sets of the corresponding $\Lambda$-multi-twisted codes.  A trace description for these codes is  given  and a lower bound on their minimum Hamming distances is also obtained by extending the work of G\"{u}neri et al. \cite{cem} to $\Lambda$-multi-twisted codes.

 This paper is organized as follows: In Section \ref{prelim}, we state some basic results from groups and geometry that are needed to obtain enumeration formulae for self-dual, self-orthogonal and complementary-dual $\Lambda$-multi-twisted codes. In Section \ref{multi-twisted}, we  study $\Lambda$-multi-twisted codes over finite fields and their dual codes with respect to the standard inner product. In Section \ref{selfdual},  we study the algebraic structure of all self-dual and self-orthogonal  $\Lambda$-multi-twisted codes. We also derive necessary and sufficient conditions for the existence of a self-dual $\Lambda$-multi-twisted code and provide enumeration formulae for each of the two aforementioned classes of $\Lambda$-multi-twisted codes (Theorems \ref{enum1} and \ref{countselfortho}).  In Section \ref{Gen},  we obtain the parity-check polynomial for a $\rho$-generator $\Lambda$-multi-twisted code and a BCH type lower bound on their minimum Hamming distances (Theorems \ref{parity0} and \ref{parity1}). We also determine generating sets of dual codes of some $\rho$-generator  $\Lambda$-multi-twisted codes from generating sets of the corresponding $\Lambda$-multi-twisted codes (Theorem \ref{parity2}). We also obtain a lower bound on the dimension of some $[\Lambda,\Omega]$-multi-twisted codes of length $n$ over $\mathbb{F}_{q},$ where $\Lambda \neq \Omega$ (Theorem \ref{t23}). We also derive some sufficient conditions for a $\Lambda$-multi-twisted code to be  LCD  (Theorems \ref{L1} and \ref{L2}).  In Section \ref{Traceformula}, we provide a trace description for $\Lambda$-multi-twisted codes by viewing these codes as direct sums of certain concatenated codes, which leads to  a method to construct these codes and a lower bound on their minimum Hamming distances (Theorems \ref{trace} and \ref{bound}). 
\section{Some preliminaries}\label{prelim}
To enumerate all self-dual and self-orthogonal $\Lambda$-multi-twisted codes over finite fields,  we need some basic results from groups and geometry, which are as discussed below:

Let $V$ be a finite-dimensional vector space over the finite field $F$ and let $B$ be a  $\sigma$-sesquilinear form on $V,$ where $\sigma$ is an automorphism of $F.$ Then the pair $(V,B)$ is called a formed space. From now on, throughout this section, we suppose that $B$ is a reflexive and non-degenerate $\sigma$-sesquilinear form on $V.$  The formed space $(V,B)$ is called (i) a symplectic space if $B$ is an alternating form on $V,$ (ii) a unitary space if $B$ is a Hermitian form on $V,$ and (iii) an orthogonal space (or a finite geometry) if $B$ is a  symmetric form on $V.$  Further, a subspace of  the formed space $(V,B)$ is defined as a pair $(U,B_U)$, where $U$ is a subspace of $V$ and $B_U=B\restriction_{U \times U}.$  Let us define $U^{\perp}=\{v \in V: B(u,v)=0 \text{ for all }u \in U\}.$
 
\begin{thm}\cite{grove, taylor} \label{dimgen} If $(V,B)$ is a finite-dimensional reflexive and non-degenerate  space over the field $F$ and $U$ is a subspace of $V,$ then $U^{\perp}$ is a subspace of $V$ and $\text{dim}_{F}U^{\perp}=\text{dim}_{F}V-\text{dim}_{F}U.$  \end{thm}
 A subspace $U$ of $V$ is said to be (i) self-dual if it satisfies $U=U^{\perp},$  (ii)  self-orthogonal (or totally isotropic) if it satisfies $U \subseteq U^{\perp},$  (iii) non-degenerate (LCD) if it satisfies $U \cap U^{\perp}=\{0\},$ and (iv) dual-containing if $U^{\perp}\subseteq U.$  The Witt index of $V$ is defined as the dimension of a maximal  self-orthogonal subspace of $V.$  
 
 Next let $\mathbb{F}_{q}^{\mu}$ be the vector space consisting of all $\mu$-tuples over the finite field $\mathbb{F}_{q}.$ With respect to the standard inner product on $\mathbb{F}_{q}^{\mu},$ the following  hold.
  \begin{thm} \label{9}\begin{enumerate}\item[(a)] \cite[Th. 9.1.3]{huffbook} There exists a self-dual subspace (code) of even length $\mu$ over $\mathbb{F}_{q}$ if and only if $(-1)^{\mu/2}$ is a square in $\mathbb{F}_{q}.$ Furthermore, if $\mu$ is even and $(-1)^{\mu/2}$ is not a square in $\mathbb{F}_{q},$ then the dimension of a maximal self-orthogonal subspace  of length $\mu$ over $\mathbb{F}_{q}$ is $(\mu-2)/2.$ If $\mu$ is odd, then the dimension of a maximal self-orthogonal subspace of length $\mu$ over $\mathbb{F}_{q}$ is $(\mu-1)/2.$\item[(b)] \cite[p. 217] {pless} Let $\mu$ be even and $(-1)^{\mu/2}$ be a square in $\mathbb{F}_{q}.$ Then the number of distinct self-dual subspaces of even length $\mu$ over $\mathbb{F}_{q}$ is given by $\prod\limits_{a=1}^{\frac{\mu}{2}-1} (q^a+1)$ when $q$ is even and by $\prod\limits_{a=0}^{\frac{\mu}{2}-1} (q^a+1) $ when $q$ is odd. \end{enumerate}
  \end{thm}
 In the following theorem, we state some basic properties of finite-dimensional symplectic spaces over finite fields.
\begin{thm}\cite{taylor}\label{symplectic} Let $(V,B)$ be a $\mu$-dimensional symplectic space over $\mathbb{F}_{q}.$   Then the dimension $\mu$ of $V$ is even and the following hold.
\begin{enumerate}\item[(a)] The Witt index of $V$ is $\frac{\mu}{2}.$
\item[(b)] For $0 \leq k \leq \frac{\mu}{2},$ the number of distinct $k$-dimensional self-orthogonal subspaces of $V$ is given by $\prod\limits_{a=0}^{k-1}(q^{\mu-2a}-1)/(q^{a+1}-1)={\mu/2 \brack k}_{q} \prod\limits_{a=0}^{k-1}(q^{\frac{\mu}{2}-a}+1),$ where ${\mu/2 \brack k }_q =\prod\limits_{d=1}^{k-1}(q^{\mu/2}-q^d)/(q^k-q^d)$ is the $q$-binomial coefficient.
\end{enumerate}
\end{thm}
\begin{proof}  One may refer to \cite[p. 69]{taylor} for proof of part (a), while part (b) is an Ex. 8.1 (ii) of \cite{taylor}.\end{proof}
In the following theorem, we  state some basic properties of finite-dimensional unitary spaces over finite fields.
\begin{thm}\cite{taylor} \label{unitary} Let $(V,B)$ be a $\mu$-dimensional unitary space over $\mathbb{F}_{q^2}.$ Let $\nu$ be the Witt index of $(V,B).$  Then we have the following:
\begin{enumerate}
\item[(a)] The Witt index $\nu$ of $V$ is given by $\nu=\left\{\begin{array}{cl} \frac{\mu}{2} &\text{if } \mu \text{ is even;}\vspace{1mm}\\
\frac{\mu-1}{2} & \text{if }\mu \text{ is odd.}\end{array}\right.$
\item[(b)]   For $0 \leq k \leq \nu,$ the number of distinct $k$-dimensional self-orthogonal subspaces of $V$ is given by \begin{equation*}\prod\limits_{a=\mu+1-2k}^{\mu} (q^a-(-1)^a)/ \prod\limits_{j=1}^{k}(q^{2j}-1).\end{equation*}
\end{enumerate}
\end{thm}
\begin{proof}  One may refer to \cite[p. 116]{taylor} for proof of part (a),  while  part (b) is an Ex. 10.4 of \cite{taylor}.\end{proof}
To study orthogonal spaces, let $q$ be an odd prime power and $V$ be a finite-dimensional vector space over $\mathbb{F}_{q}.$ Then the map $\varphi:V \rightarrow \mathbb{F}_q$ is called a quadratic map  on $V$ if it satisfies  (i) $\varphi(a v_1)=a^2\varphi (v_1)$ for all $a \in \mathbb{F}_{q}$ and $v_1 \in V,$ and (ii) the  map $B_{\varphi}: V \times V \rightarrow \mathbb{F}_{q},$ defined by $B_{\varphi}(v_1,v_2)=\varphi(v_1+v_2) - \varphi (v_1) -\varphi(v_2)$ for all $v_1,v_2 \in V,$ is a symmetric bilinear form on $V.$ The pair $(V,\varphi)$ is called a quadratic space over $\mathbb{F}_q.$ The quadratic space $(V,\varphi)$ over $\mathbb{F}_q$ is called non-degenerate if it satisfies $\varphi^{-1}(0)\cap V^{\perp}=\{0\},$ where $V^{\perp}=\{v\in V: B_{\varphi}(v,u)=0 \text{ for all }u \in V\}.$ If  the quadratic space $(V,\varphi)$ is non-degenerate, then the associated orthogonal space $(V,B_{\varphi})$ is called a finite geometry over $\mathbb{F}_{q}.$   On the other hand, with every symmetric bilinear form $B$ on a vector space $V$ over $\mathbb{F}_{q},$ one can associate the following quadratic map: $Q_B(v)=\frac{1}{2}B(v,v)$ for each $v \in V.$    In the following theorem, we state some basic properties of non-degenerate quadratic  spaces over a finite field of odd characteristic.
\begin{thm} \cite{pless, taylor}\label{quadratic} Let $(V,\varphi)$ be a $\mu$-dimensional non-degenerate quadratic space over the finite field $\mathbb{F}_{q}$ having an odd characteristic. Let  $\nu$ be the  Witt index of $(V,\varphi).$ Then we have the following:
\begin{enumerate}
\item[(a)] The Witt index $\nu$ of $V$ is given by \begin{equation*}\nu=\left\{\begin{array}{cl} \frac{\mu-1}{2} & \text{if }\mu \text{ is odd;}\\ \frac{\mu}{2} &\text{if } \mu \text{ is even and }q \equiv 1~(\text{mod }4) \text{ or } \mu \equiv 0~(\text{mod }4) \text{ and }  q \equiv 3~(\text{mod }4);\\\frac{\mu-2}{2} & \text{if } \mu \equiv 2~(\text{mod }4) \text{ and } q \equiv 3~(\text{mod }4).\end{array}\right.\end{equation*}
\item[(b)]   For $0 \leq k \leq \nu,$ the number of distinct $k$-dimensional self-orthogonal (or totally singular)  subspaces of $V$ is given by ${\nu \brack k }_q \prod\limits_{a=0}^{k-1}(q^{\nu-\epsilon-a}+1),$ where ${\nu \brack k }_q =\prod\limits_{d=1}^{k-1}(q^{\nu}-q^d)/(q^k-q^d)$ is the $q$-binomial coefficient and $\epsilon=2\nu-\mu+1.$ (Note that $\epsilon=1$ if $\nu=\frac{\mu}{2},$ $\epsilon=-1$ if $\nu=\frac{\mu-2}{2}$ and $\epsilon=0$ if $\nu=\frac{\mu-1}{2}.$)
 \end{enumerate}
\end{thm}
\begin{proof}  As $q$ is odd,  the orthogonal space $(V,B_{\varphi})$ associated with $(V,\varphi)$  is a finite geometry over $\mathbb{F}_{q}$ having an orthogonal basis. Now by applying Theorem 1 of Pless \cite{pless}, part (a) follows immediately.  Part (b) is an Ex. 11.3 of \cite{taylor}.\end{proof}
%
\section{Multi-twisted (MT) codes over finite fields and their dual codes}\label{multi-twisted}
Throughout this paper, let $\mathbb{F}_{q}$ denote the finite field of order $q,$ let $m_1,m_2,\cdots,m_{\ell}$ be positive integers coprime to $q,$ and let $n=m_1+m_2+\cdots+m_{\ell}.$ Let $\mathbb{F}_{q}^n$ denote the vector space consisting of all $n$-tuples over $\mathbb{F}_{q}.$ Let $\Lambda=(\lambda_1,\lambda_2,\cdots,\lambda_{\ell})$ and $\Lambda'=(\lambda_1^{-1},\lambda_2^{-1},\cdots,\lambda_{\ell}^{-1}),$ where $\lambda_1,\lambda_2,\cdots,\lambda_{\ell}$ are non-zero elements of $\mathbb{F}_{q}.$  Then a $\Lambda$-multi-twisted  module $V$ is an $\mathbb{F}_{q}[x]$-module of the form $V=\prod\limits_{i=1}^{\ell}V_i,$ where $V_i=\mathbb{F}_{q}[x]/\langle x^{m_i}-\lambda_i \rangle$ for $1 \leq i \leq \ell.$ We note that there exists an $\mathbb{F}_{q}$-linear vector space isomorphism from $\mathbb{F}_{q}^n$ onto $V.$ We shall represent each element $a \in \mathbb{F}_{q}^n$ as $a=(a_{1,0},a_{1,1},\cdots,a_{1,m_1-1}; a_{2,0},a_{2,1},\cdots, a_{2,m_2-1}; \cdots; a_{\ell,0}, a_{\ell,1},\cdots, a_{\ell,m_{\ell}-1}) $ and the corresponding element $a(x) \in V$  as $a(x) =(a_1(x),a_2(x),\cdots,a_{\ell}(x)),$ where $a_i(x)=\sum\limits_{j=0}^{m_i-1}a_{i,j}x^j \in V_i $ for $1 \leq i \leq \ell.$    Now a $\Lambda$-multi-twisted (MT) code of length $n$ over $\mathbb{F}_{q}$ is defined as an $\mathbb{F}_{q}[x]$-submodule of the $\Lambda$-multi-twisted module $V.$ Equivalently, a linear code $\mathcal{C}$ of length $n$ over $\mathbb{F}_{q}$ is called a $\Lambda$-multi-twisted code if $c=(c_{1,0},c_{1,1},\cdots,c_{1,m_1-1}; c_{2,0},c_{2,1},\cdots,\textit{•} c_{2,m_2-1}; \cdots; c_{\ell,0}, c_{\ell,1},\cdots, c_{\ell,m_{\ell}-1}) \in \mathcal{C},$ then its $\Lambda$-multi-twisted shift $T_{\Lambda}(c)=(\lambda_{1}c_{1,m_1-1},c_{1,0},\cdots,c_{1,m_1-2}; \lambda_2 c_{2,m_2-1},c_{2,0},\cdots, c_{2,m_2-2}; \cdots; \lambda_{\ell} c_{\ell,m_{\ell}-1}, c_{\ell,0},\cdots, c_{\ell,m_{\ell}-2}) $ is also a codeword of $\mathcal{C}.$ In particular, when $m_1=m_2=\cdots=m_{\ell}$ and $\lambda_1=\lambda_2=\cdots=\lambda_{\ell},$ $\Lambda$-multi-twisted codes are permutation-equivalent to quasi-twisted codes of length $m_1 \ell$ over $\mathbb{F}_{q}.$
When $\lambda_i=1$ for $1 \leq i \leq \ell,$  $\Lambda$-multi-twisted codes  coincide with generalized quasi-cyclic (GQC) codes, which are first  defined and studied by Siap and Kulhan \cite{Siap}. Furthermore, when $m_1=m_2=\cdots=m_{\ell}$ and $\lambda_i=1$ for $1 \leq i \leq \ell,$ $\Lambda$-multi-twisted codes are permutation-equivalent to quasi-cyclic (QC) codes of  length $m_1 \ell$ and index $\ell$ over $\mathbb{F}_{q}.$ Besides this, when $\ell=1,$  $\Lambda$-multi-twisted codes are $\lambda_1$-constacyclic codes of length $m_1$ over $\mathbb{F}_{q}.$

To study the algebraic structure of $\Lambda$-multi-twisted codes, let $g_1(x),g_2(x),\cdots,g_r(x)$ be all the distinct irreducible factors of the polynomials $x^{m_1}-\lambda_1,x^{m_2}-\lambda_2, \cdots, x^{m_{\ell}}-\lambda_{\ell} $ over $\Fq.$ For $1 \le w \le r$ and $1 \le i \le \ell,$ let us define $$\epsilon_{w,i}=\left\{\begin{array}{ll} 1 & \text{if } g_{w}(x)\text{ divides } x^{m_i} - \lambda_i \text{ in }\mathbb{F}_{q}[x];\\0 & \text{ otherwise.}\end{array}\right.$$
Then for $1 \leq i \leq \ell,$ we note that $x^{m_i}-\lambda_i=\prod\limits_{w=1}^r g_w(x)^{\epsilon_{w,i}}$ is the irreducible factorization of $x^{m_i}-\lambda_i$ over $\mathbb{F}_{q}.$ Now for each $i,$ by applying Chinese Remainder Theorem, we get \vspace{-2mm}\begin{equation*} V_i \simeq  \bigoplus_{w=1}^{r} \epsilon_{w,i}F_w \vspace{-2mm}\end{equation*} with $F_w=\frac{\mathbb{F}_q[x]}{\langle g_w(x)\rangle}$ for $1 \leq w \leq r;$  the corresponding ring isomorphism from $V_i$ onto $\bigoplus\limits_{w=1}^{r} \epsilon_{w,i}F_w$  is given by $a_i(x)\mapsto \sum\limits_{w=1}^r\Big(\epsilon_{w,i}a_i(x)+\langle g_w(x) \rangle \Big)$ for each $a_i(x) \in V_i.$  This further induces a ring isomorphism from $V$ onto $\bigoplus\limits_{w=1}^{r} \Big( \underbrace{ \epsilon_{w,1} F_w, \epsilon_{w,2} F_w,\cdots , \epsilon_{w,\ell} F_w }_{\mathcal{G}_{w}}\Big),$ which is given by $(a_1(x),a_2(x),\cdots,a_{\ell}(x))\mapsto \sum\limits_{w=1}^r \Big(\epsilon_{w,1}a_1(x)+\langle g_w(x)\rangle ,\epsilon_{w,2}a_2(x)\\+\langle g_w(x) \rangle,\cdots,\epsilon_{w,{\ell}}a_{\ell}(x)+\langle g_w(x) \rangle\Big)$ for each $(a_1(x),a_2(x),\cdots,a_{\ell}(x)) \in V.$ 
 If  $d_w=\text{deg }g_w(x),$ then we see that  $F_w \simeq \mathbb{F}_{q^{d_w}}$  for $1 \leq w \leq r.$ Next let $\epsilon_w=\sum\limits_{i=1}^{\ell} \epsilon_{w,i}$  for each $w.$
It is easy to see that for $1 \leq w \leq r,$ $\mathcal{G}_w=\Big(\epsilon_{w,1} F_w, \epsilon_{w,2} F_w,\cdots , \epsilon_{w,\ell} F_w \Big)$ is an $\epsilon_{w}$-dimensional vector space over $F_w.$ As $\mathcal{G}_w \subset F_w^{\ell}, $  a $\Lambda$-multi-twisted code can be viewed as a submodule of $ \bigoplus\limits_{w=1}^{r} F_w^{\ell}$ over $\bigoplus\limits_{w=1}^{r}F_w.$  From the above discussion, we deduce the following:
{\thm \label{t11} Let $\mathcal{C}$ be a $\Lambda$-multi-twisted code of length $n$ over $\mathbb{F}_{q},$ which is finitely-generated as an $\mathbb{F}_{q}[x]$-submodule of $V$ by
$\{(a_{d,1}(x),a_{d,2}(x),\cdots,a_{d,{\ell}}(x)):1\leq d \leq k \}\subset V$. Then the code $\mathcal{C}$ can be uniquely expressed as $\mathcal{C}=\bigoplus\limits_{w=1}^r \mathcal{C}_w,$  where for $1 \leq w \leq r,$ $\mathcal{C}_w$ is a linear code of length $\ell$ over $F_w,$ given by  \begin{equation*}\mathcal{C}_w = \text{Span}_{F_w}\{(\epsilon_{w,1}a_{d,1}(\alpha_w),\epsilon_{w,2}a_{d,2}(\alpha_w),\cdots,\epsilon_{w,{\ell}}a_{d,{\ell}}(\alpha_w)):1 \leq d \leq k\} \end{equation*} with  $\alpha_w$ as a zero of $g_w(x)$ in $F_w,$  (the codes $\mathcal{C}_{1},$ $\mathcal{C}_{2}, \cdots, \mathcal{C}_{r}$ are called the constituents of $\mathcal{C}$). Furthermore, we have  $\text{dim}_{\mathbb{F}_{q}}\mathcal{C}=\sum\limits_{w=1}^{r}\text{dim}_{F_w}\mathcal{C}_{w} \text{deg }g_w(x).$
Conversely, if $\mathcal{D}_{w}$ is a linear code of length $\ell$ over $F_w$ for $1 \leq w \leq r,$ then $\mathcal{D}=\bigoplus\limits_{w=1}^{r}\mathcal{D}_{w}$ is a $\Lambda$-multi-twisted code of length $n$ over $\mathbb{F}_{q}.$}

Next to enumerate all $\Lambda$-multi-twisted codes, we recall that for non-negative integers $b,k$ with $b \leq k$ and a prime power $Q,$ the number  of distinct $b$-dimensional subspaces of a $k$-dimensional vector space over $\mathbb{F}_{Q}$ is given  by the $Q$-binomial coefficient ${k\brack b}_Q=\frac{(Q^k-1)(Q^k-Q)\cdots(Q^k-Q^{b-1})}{(Q^b-1)(Q^b-Q)\cdots(Q^b-Q^{b-1})}.$This  implies that the total number of distinct subspaces of a $k$-dimensional vector space over $\mathbb{F}_{Q}$ is given by \begin{equation}\label{sub}N(k,Q)=1 + \sum\limits_{b=1}^k {k\brack b}_Q.\end{equation}
In the following theorem, we enumerate all $\Lambda$-multi-twisted codes of length $n$ over $\mathbb{F}_{q}.$
 {\thm Let $\lambda_1,\lambda_2,\cdots,\lambda_{\ell} \in \Fq \setminus \{0\}$ be fixed. Then the total number of distinct $\Lambda$-multi-twisted codes  of length $n$ over $\mathbb{F}_{q}$ is given by 
$N_{\Lambda}=\prod\limits_{w=1}^r \left( 1 + \sum\limits_{b=1}^{\epsilon_{w}} { \epsilon_{w} \brack b}_{q^{d_{w}}}\right),$ where  $d_{w}=\text{deg }g_{w}(x)$ for each $w.$ \begin{proof} By Theorem \ref{t11}, we see that all the  $\Lambda$-multi-twisted codes of length $n$ over $\mathbb{F}_{q}$ are of the form $\bigoplus\limits_{w=1}^r \mathcal{C}_w,$  where $\mathcal{C}_w$ is a linear code of length $\ell$ over $F_w$ for $1 \leq w \leq r.$ Now using \eqref{sub} and  the fact that $F_w \simeq \mathbb{F}_{q^{d_{w}}},$   the desired result follows immediately.   \end{proof}}
\begin{rem}It is easy to see that some $\Lambda$-multi-twisted codes  can also be viewed as $\Omega$-multi-twisted codes, where $\Omega \neq \Lambda.$ For example, when $q=7,$ $m_1=2$ and $m_2=1,$ the linear code $\mathcal{C}$ with the basis set as $\{(1,0;0), (0,1;0)\}$ is a $(2,1)$-multi-twisted as well as $(4,1)$-multi-twisted code of length 3 over $\mathbb{F}_{7}.$ Thus the total number of  distinct multi-twisted codes of length $n$ over $\mathbb{F}_{q}$ is not equal to $(q-1)^{\ell}N_{\Lambda}.$ \end{rem}

Next we shall study dual codes of $\Lambda$-multi-twisted codes with respect to the  standard inner product on $\mathbb{F}_{q}^n,$ which is a map $\langle\cdot,\cdot\rangle: \Fqn \times \Fqn \longrightarrow \mathbb{F}_q,$ defined as $\langle x, y \rangle =\sum\limits_{i=1}^{\ell} \sum\limits_{j=0}^{m_i-1} x_{i,j}y_{i,j}$ for all $x,y \in \mathbb{F}_{q}^n.$ It is well-known that $\left<\cdot,\cdot\right>$ is a non-degenerate and symmetric bilinear form on $\mathbb{F}_{q}^n.$  If $\mathcal{C}$ is a $\Lambda$-multi-twisted code of length $n$ over $\mathbb{F}_{q},$ then its dual code $\mathcal{C}^{\perp}$ is defined as  $\mathcal{C}^{\perp} =
\{ a  \in \mathbb{F}_{q}^n : \langle a, c \rangle=0 \text{ for all } c\in \mathcal{C}\}.$ It is easy to see that $\mathcal{C}^{\perp}$ is a $\Lambda'$-multi-twisted code of length $n$ over $\mathbb{F}_{q},$ i.e., $\mathcal{C}^{\perp}$ is a linear code of length $n$ over $\mathbb{F}_{q}$ satisfying the following:
if $d=(d_{1,0},d_{1,1},\cdots,d_{1,m_1-1}; d_{2,0},d_{2,1},\cdots, d_{2,m_2-1}; \cdots; d_{\ell,0}, d_{\ell,1},\cdots, d_{\ell,m_{\ell}-1}) \in \mathcal{C}^{\perp},$ then its $\Lambda'$-multi-twisted shift $T_{\Lambda'}(d)=(\lambda_{1}^{-1}d_{1,m_1-1},d_{1,0},\cdots,d_{1,m_1-2}; \lambda_2^{-1} d_{2,m_2-1},d_{2,0},\cdots, d_{2,m_2-2}; \cdots; \lambda_{\ell}^{-1} d_{\ell,m_{\ell}-1}, d_{\ell,0},\cdots, d_{\ell,m_{\ell}-2})  \in \mathcal{C}^{\perp}.$
Equivalently, $\mathcal{C}^{\perp}$ is an $\mathbb{F}_{q}[x]$-submodule of the $\Lambda'$-multi-twisted module $V'=\prod\limits_{i=1}^{\ell} V_i',$ where $V_i'=\mathbb{F}_{q}[x]/\langle x^{m_i}-\lambda_i^{-1}\rangle$ for $1 \leq i \leq \ell.$

To study their  dual codes in more detail, let $m$ be the order of the polynomial $\text{lcm}[x^{m_1}-\lambda_1,x^{m_2}-\lambda_2,\cdots,x^{m_{\ell}}-\lambda_{\ell}] $ in $\mathbb{F}_{q}[x].$ It is easy to observe that $m=\text{lcm}[m_1O(\lambda_1),m_2O(\lambda_2),\cdots,m_{\ell}O(\lambda_{\ell})]$ and that $T_{\Lambda}^m=T_{\Lambda'}^m=I,$ where $I$ is the identity operator on $\mathbb{F}_{q}^n$ and $O(\lambda_i)$ denotes the multiplicative order of $\lambda_i$ for each $i.$ For $1 \leq i \leq \ell,$ define a conjugation map $\overline{ ^{\hspace{2mm}} }: V_i' \rightarrow V_i $ as  $\overline{b_i(x)}= b_i(x^{-1})$ for each $b_i(x) \in V_i',$ where  $x^{-1}=\lambda_i^{-1}x^{m_i-1} \in V_i.$ Next
 we define a mapping $\left(\cdot,\cdot\right):V \times V' \longrightarrow \frac{\mathbb{F}_{q}[x]}{\langle x^m-1\rangle}$ as $\left(a(x),b(x)\right) :=\sum\limits_{i=1}^{\ell} \lambda_ia_i(x)\overline{b_i(x)} \Big( \frac{x^m -1}{x^{m_i} - \lambda_i}\Big)$ for 
 $a(x)= (a_1(x),a_2(x),\cdots,a_{\ell}(x)) \in V$ and $ b(x)= (b_1(x),b_2(x),\cdots,b_{\ell}(x)) \in V',$ where $\frac{\mathbb{F}_{q}[x]}{\langle x^m-1 \rangle}$ is viewed as an $\mathbb{F}_{q}[x]$-module.  Then we have the following:

{\lem\label{l4} \begin{enumerate}\item[(a)]  For $a(x) \in V$ and $b(x) \in V',$ we have
$$\begin{array}{ll}
\left(a(x),b(x)\right) & = \left< a , b\right> + \left< a , T_{\Lambda'}(b) \right>x + \cdots + \left< a , T_{\Lambda'}^{m-1}(b) \right> x^{m-1}\\
 & = \left< b , a \right> + \left<b , T_{\Lambda}^{m-1}(a) \right>x + \cdots + \left<b , T_{\Lambda}(a) \right> x^{m-1} \text{ in } \frac{\mathbb{F}_{q}[x]}{\langle x^m-1 \rangle}.\end{array}$$
 \item[(b)] The mapping $\left(\cdot,\cdot \right)$ is a non-degenerate  and Hermitian $\overline{ ^{\hspace{2mm}} }$-sesquilinear form on $V \times V'.$ \end{enumerate}
 \begin{proof} 
\begin{enumerate}\item[(a)] To prove this, we first write $a(x)=(a_1(x),a_2(x),\cdots,a_{\ell}(x))$ and $b(x)=(b_1(x),b_2(x),\cdots,b_{\ell}(x)),$ where $a_i(x) =\sum\limits_{j=0}^{m_i-1}a_{i,j}x^j\in V_i$ and $b_i(x)=\sum\limits_{j=0}^{m_i-1}b_{i,j}x^j \in V_i' $ for each $i.$ For $1 \leq i \leq \ell,$ we observe that $\frac{\lambda_i(x^m -1)}{x^{m_i} - \lambda_i} = 1 + \lambda_i^{-1}x^{m_i} + \lambda_i^{-2}x^{2m_i} + \cdots + \lambda_i^{-(\frac{m}{m_i} -2)}x^{(\frac{m}{m_i} -2)m_i} + \lambda_i x^{(\frac{m}{m_i} -1)m_i}.$ Using this, we get $\left(a(x),b(x)\right) = \left< {a} , {b} \right> + \left< {a} , T_{\Lambda'}({b}) \right>x + \cdots + \left< {a} , T_{\Lambda'}^{m-1}({b}) \right> x^{m-1}.$
As $\left<{a},T_{\Lambda'}^j({b})\right>=\left<{b},T_{\Lambda}^{m-j}({a})\right>$ for $0\leq j \leq m-1,$ we get $\left({a}(x),{b}(x)\right)= \left< {b} , {a} \right> + \left< {b} , T_{\Lambda}^{m-1}({a}) \right>x + \cdots + \left< {b} , T_{\Lambda}({a}) \right> x^{m-1} .$

\item[(b)] It is easy to observe that $\left(\cdot,\cdot \right)$ is a Hermitian $\overline{ ^{\hspace{2mm}} }$-sesquilinear form on $V \times V'.$ To prove the non-degeneracy of $\left( \cdot,\cdot \right),$  suppose that $\left({a}(x),{b}(x)\right)=0$ for all ${b}(x)\in V'.$ Here we need to show that ${a}(x)=0.$ For this, by part (a), we see that  $\left({a}(x),{b}(x)\right)=\left< {a} , {b} \right> + \left< {a} , T_{\Lambda'}({b}) \right>x + \cdots + \left< {a} , T_{\Lambda'}^{m-1}({b}) \right> x^{m-1}=0$ for all ${b} \in \mathbb{F}_{q}^n.$ This implies that $\left< {a} , {b} \right> = \left< {a} , T_{\Lambda'}({b}) \right> =  \cdots = \left< {a} , T_{\Lambda'}^{m-1}({b}) \right>=0$ for all ${b}\in \mathbb{F}_{q}^n.$ As $\left<\cdot, \cdot \right>$ is a non-degenerate bilinear form on $\mathbb{F}_{q}^n,$ we get ${a}=0,$ which gives ${a}(x)=0.$ This proves (b). \end{enumerate}\vspace{-5mm}\end{proof}}

From the above, we deduce the following:

\begin{prop} If $\mathcal{C} \subseteq V$ is a $\Lambda$-multi-twisted code of length $n$ over $\mathbb{F}_{q},$ then  its dual code $\mathcal{C}^{\perp} \subseteq V'$ is a $\Lambda'$-multi-twisted code  of length $n$ over $\mathbb{F}_{q}$ and  is given by $\mathcal{C}^{\perp}=\{ b(x) \in V': \left(a(x),b(x)\right)=0 \text{ for all }a(x) \in \mathcal{C}\}.$ \end{prop}

From now on, we shall follow the same notations as in Section \ref{multi-twisted}.   
\section{Self-dual and self-orthogonal $\Lambda$-multi-twisted codes} \label{selfdual}
A $\Lambda$-multi-twisted  code $\mathcal{C}$ of  length $n$ over $\mathbb{F}_{q}$ is said to be (i) self-dual if it satisfies $\mathcal{C}=\mathcal{C}^{\perp}$ and (ii) self-orthogonal if it satisfies $\mathcal{C} \subseteq \mathcal{C}^{\perp}.$ These two classes of  $\Lambda$-multi-twisted codes have nice algebraic structures and are useful in constructing modular forms.  Now we proceed to study the algebraic structures of  self-dual and self-orthogonal $\Lambda$-multi-twisted codes of length $n$ over $\mathbb{F}_{q}.$

To do this, for a non-zero polynomial $f(x)$ over $\mathbb{F}_{q},$ let $f^*(x)$ denote its reciprocal polynomial.  Further, the polynomial $f(x)(\neq 0) \in \mathbb{F}_{q}[x]$ is said to be self-reciprocal if it satisfies $\langle f(x)\rangle=\langle f^*(x) \rangle $  in $\mathbb{F}_{q}[x].$ Two non-zero polynomials $f(x),g(x) \in \mathbb{F}_{q}[x]$ form a reciprocal pair if  they satisfy $\langle g(x) \rangle =\langle f^*(x) \rangle $  in $\mathbb{F}_{q}[x].$ Now we recall that $g_1(x), g_2(x),\cdots, g_r(x)$ are all the distinct irreducible factors of $x^{m_1}-\lambda_1, x^{m_2}-\lambda_2,\cdots, x^{m_{\ell}}-\lambda_{\ell}$ in $\mathbb{F}_{q}[x]$ with  $\text{deg }g_w(x)=d_w$ for $1 \leq w \leq r.$  As $g_w(x)$ is irreducible over $\mathbb{F}_{q},$ we see that $\text{deg }g_w^*(x)=\text{deg }g_w(x)=d_w$ for each $w.$ Further, suppose (by relabelling $g_w(x)$'s if required) that $g_1(x),g_2(x),\cdots, g_s(x)$ are all the distinct self-reciprocal polynomials,  $g_{s+1}(x),g_{s+1}^*(x),\cdots,g_t(x),g_t^*(x)$ are all the polynomials forming reciprocal pairs, and that $g_{t+1}(x),g_{t+2}(x),\cdots, g_{e}(x)$ are the remaining polynomials (that are neither self-reciprocal nor do they form reciprocal pairs),  which   appear in the irreducible factorizations of $x^{m_1}-\lambda_1,x^{m_2}-\lambda_2,\cdots,x^{m_{\ell}}-\lambda_{\ell}$ over $\Fq.$ Here $r=e+t-s.$ Next for $1 \leq u \leq s,$  $s+1 \leq v \leq t$ and $t+1 \leq {p} \leq e,$ we note that $F_u= \frac{\Fq[x]}{\left<g_u(x)\right>} \simeq \mathbb{F}_{q^{d_u}},$ $F_v= \frac{\Fq[x]}{\left<g_v(x)\right>}\simeq \mathbb{F}_{q^{d_v}},$ $F_v'= \frac{\Fq[x]}{\left<g_v^*(x)\right>}\simeq \mathbb{F}_{q^{d_v}},$ $F_{p} =\frac{\mathbb{F}_{q}[x]}{\langle g_{p}(x) \rangle} \simeq \mathbb{F}_{q^{d_{p}}}$ and $F_{p}' =\frac{\mathbb{F}_{q}[x]}{\langle g_{p}^*(x) \rangle} \simeq \mathbb{F}_{q^{d_{p}}}.$
Therefore by Chinese Remainder Theorem, we have 
\\$  V\simeq \Big( \bigoplus\limits_{u=1}^s \bigg( \underbrace{\epsilon_{u,1}F_u, \epsilon_{u,2}F_u, \cdots, \epsilon_{u,\ell}F_u}_{\mathcal{G}_u}\bigg) \Big) 
\oplus \Big( \bigoplus\limits_{v=s+1}^t \bigg\{ \bigg( \underbrace{\epsilon_{v,1}F_v, \epsilon_{v,2}F_v, \cdots, \epsilon_{v,\ell}F_v}_{\mathcal{G}_v}\bigg)\oplus \bigg( \underbrace{\epsilon'_{v,1}F_v', \epsilon'_{v,2}F_v', \cdots, \epsilon'_{v,\ell}F_v'}_{\mathcal{G}_v'} \bigg) \bigg\} \Big)$\\ \small{\begin{equation}\label{51}
 \oplus  \Big( \bigoplus\limits_{p=t+1}^{e} \bigg\{ \bigg( \underbrace{\epsilon_{p,1}F_{p}, \epsilon_{p,2}F_{p}, \cdots, \epsilon_{p,\ell}F_{p}}_{\mathcal{G}_{p}}\bigg)\end{equation}}\normalsize
and \\ $  V'\simeq \Big( \bigoplus\limits_{u=1}^s \bigg( \underbrace{\epsilon_{u,1}F_u, \epsilon_{u,2}F_u, \cdots, \epsilon_{u,\ell}F_u}_{\mathcal{G}_u}\bigg) \Big) 
\oplus \Big( \bigoplus\limits_{v=s+1}^t \bigg\{ \bigg( \underbrace{\epsilon'_{v,1}F_v, \epsilon'_{v,2}F_v, \cdots, \epsilon'_{v,\ell}F_v}_{\mathcal{H}_v}\bigg)\oplus \bigg( \underbrace{\epsilon_{v,1}F_v', \epsilon_{v,2}F_v', \cdots, \epsilon_{v,\ell}F_v'}_{\mathcal{H}_v'} \bigg) \bigg\} \Big)$\\ \small{\begin{equation}\label{51}
 \oplus  \Big( \bigoplus\limits_{p=t+1}^{e} \bigg\{ \bigg( \underbrace{\epsilon_{p,1}F_{p}', \epsilon_{p,2}F_{p}', \cdots, \epsilon_{p,\ell}F_{p}'}_{\mathcal{G}_{p}'}\bigg),\end{equation}}\normalsize
where for $1 \leq {\alpha} \leq e,$  $s+1\leq v \leq t$ and $1 \leq i \leq \ell,$ \begin{equation*}\epsilon_{\alpha,i}=\left\{\begin{array}{ll} 1 & \text{if }g_\alpha(x)  \text{ divides }x^{m_i}-\lambda_i \text{ in }\mathbb{F}_{q}[x];\\ 0 & \text{otherwise} \end{array}\right. \text{ and  }\epsilon'_{v,i}=\left\{\begin{array}{ll} 1 & \text{if }g_v^*(x)  \text{ divides }x^{m_i}-\lambda_i \text{ in }\mathbb{F}_{q}[x];\\ 0 & \text{otherwise.} \end{array}\right.\end{equation*}
Note that $\text{dim}_{F'_v}\mathcal{H}_{v}'=\epsilon_v$ for each $v.$
Further, if $\epsilon'_v=\sum\limits_{i=1}^{\ell}\epsilon'_{v,i},$ then  $\text{dim}_{F_v'}\mathcal{G}_{v}'=\text{dim}_{F_v}\mathcal{H}_{v}=\epsilon'_v$ for each $v.$ We also recall that $\text{dim}_{F_{\alpha}}\mathcal{G}_{\alpha}=\epsilon_{\alpha} =\sum\limits_{i=1}^{\ell}\epsilon_{\alpha,i}$ for $1 \leq {\alpha} \leq e.$
In view of this, from now on, we shall identify each element ${a}(x)=(a_1(x),a_2(x),\cdots,a_{\ell}(x)) \in V$ as  $A=(A_1,A_2,\cdots,{A_s},{A_{s+1}},{A_{s+1}'}, \cdots,{A_t},{A_t'}, A_{t+1},\cdots, A_e),$ where ${A_u}=(A_{u,1},A_{u,2},\cdots,A_{u, \ell}) \in \mathcal{G}_{u},$  ${A_v}=(A_{v,1},A_{v,2},
\cdots,A_{v,\ell})\in \mathcal{G}_{v},$ ${A_v'}=(A_{v,1}',A_{v,2}',\cdots,A_{v,\ell}')\in \mathcal{G}_{v}'$ and $A_e=(A_{e,1},A_{e,2},\cdots,A_{e,\ell}) \in \mathcal{G}_{e}$  with $A_{u,i}:=\epsilon_{u,i}a_i(x)+\left<g_u(x)\right>,$ $A_{v,i}:=\epsilon_{v,i}a_i(x)+\left<g_v(x)\right>,$  $A_{v,i}':=\epsilon'_{v,i}a_i(x)+\left<g_v^*(x)\right>$ and $A_{p,i}=\epsilon_{p,i}a_i(x)+\left<g_{p}(x)\right>$  for $1 \leq i \leq \ell,$  $1 \leq u \leq s,$  $s+1 \leq v \leq t$ and $t+1 \leq {p} \leq e.$ Analogously, we shall identify each element $b(x)=(b_1(x),b_2(x),\cdots,b_{\ell}(x)) \in V'$ as $B=(B_1,{B_2},\cdots,{B_s},{B_{s+1}},{B'_{s+1}}, \cdots,{B_t},{B'_t}, {B_{t+1}}, \cdots, {B_e}),$ where ${B_u}=({B_{u,1}},{B_{u,2}},\cdots,{B_{u, \ell}}) \in \mathcal{G}_{u},$  ${B_v}=(B_{v,1},{B_{v,2}},\cdots,{B_{v,\ell}})\in \mathcal{H}_{v},$ ${B_v'}=({B_{v,1}'},{B_{v,2}'},\cdots,{B_{v,\ell}'})\in \mathcal{H}'_{v}$ and ${B_{p}}=({B_{p,1}},{B_{p,2}},\cdots,{B_{p,\ell}}) \in \mathcal{G}'_{p}$  with ${B_{u,i}}:=\epsilon_{u,i}{b_i(x)}+\left<g_u(x)\right>,$ ${B_{v,i}}:=\epsilon'_{v,i}{b_i(x)}+\left<g_v(x)\right>,$  ${B_{v,i}'}:=\epsilon_{v,i}{b_i(x)}+\left<g_v^*(x)\right>$  and ${B_{p,i}}=\epsilon_{p,i}{b_i(x)}+\left<g_{p}^*(x)\right>$ for each $i,u,v$ and $p.$ 
    For  $1 \leq u \leq s,$ let $\overline{{ }^{\hspace{2mm} }}: \epsilon_{u,i}F_{u} \rightarrow \epsilon_{u,i}F_{u}$ be the conjugation map, defined as \begin{equation*}\overline{h_{u}(x)}=\left\{\begin{array}{cl} h_{u}(x^{-1})=h_{u}(\lambda_{i}^{-1} x^{m_{i}-1}) & \text{if }\epsilon_{u,i}=1;\\ 0 & \text{if } \epsilon_{u,i}=0\end{array}\right.\end{equation*}for all $h_{u}(x) \in \epsilon_{u,i} F_{u}.$
    For  $s+1 \leq v \leq t,$ the conjugation map $\overline{{ }^{\hspace{2mm} }}: \epsilon_{v,i}'F_{v} \rightarrow \epsilon_{v,i}'F_{v}'$ is defined as \begin{equation*}\overline{h_{v}(x)}=\left\{\begin{array}{cl} h_{v}(x^{-1})=h_{v}(\lambda_{i}^{-1} x^{m_{i}-1}) & \text{if }\epsilon_{v,i}'=1;\\ 0 & \text{if } \epsilon_{v,i}'=0\end{array}\right.\end{equation*}    
 for all $h_{v}(x) \in \epsilon_{v,i}'F_{v},$ while the conjugation map $\overline{{ }^{\hspace{2mm} }}: \epsilon_{v,i}F_{v}' \rightarrow \epsilon_{v,i}F_{v}$ is defined as \begin{equation*}\overline{\hat{h}_{v}(x)}=\left\{\begin{array}{cl} \hat{h}_{v}(x^{-1})=\hat{h}_{v}(\lambda_{i}^{-1} x^{m_{i}-1}) & \text{if }\epsilon_{v,i}=1;\\ 0 & \text{if } \epsilon_{v,i}=0\end{array}\right.\end{equation*}    
 for all $\hat{h}_{v}(x) \in \epsilon_{v,i}F_{v}'.$ For  $t+1 \leq {p} \leq e,$ the conjugation map $\overline{{ }^{\hspace{2mm} }}: \epsilon_{p,i}F_{p}' \rightarrow \epsilon_{p,i}F_{p}$ is defined as \begin{equation*}\overline{h_{p}(x)}=\left\{\begin{array}{cl} h_{p}(x^{-1})=h_{p}(\lambda_{i}^{-1} x^{m_{i}-1}) & \text{if }\epsilon_{p,i}=1;\\ 0 & \text{if } \epsilon_{p,i}=0\end{array}\right.\end{equation*}
 for all $h_{p}(x) \in \epsilon_{p,i}F_{p}'.$  For $1 \leq i \leq \ell$ and $1 \leq u \leq s$ satisfying $\epsilon_{u,i}=1,$ we observe that the conjugation map $\overline{{ }^{\hspace{2mm} }}$ is the identity map when  $d_u=1,$ while it is an automorphism of $F_u$ when $d_u >1.$
From this, we see that for each $b(x)=(b_1(x),b_2(x),\cdots,b_{\ell}(x)) \in V',$ $\overline{b(x)}\in V$ is given by $(\overline{B_1},\overline{B_2},\cdots,\overline{B_s},\overline{B_{s+1}'},\overline{B_{s+1}}, \cdots,\overline{B_t'},\overline{B_t}, \overline{B_{t+1}}, \cdots, \overline{B_e}),$ where $\overline{B_u}=(\overline{B_{u,1}},\overline{B_{u,2}},\cdots,\overline{B_{u, \ell}}) \in \mathcal{G}_{u},$  $\overline{B_v}=(\overline{B_{v,1}},\overline{B_{v,2}},\cdots,\overline{B_{v,\ell}})\in \mathcal{G}_{v}',$ $\overline{B_v'}=(\overline{B_{v,1}'},\overline{B_{v,2}'},\cdots,\overline{B_{v,\ell}'})\in \mathcal{G}_{v}$ and $\overline{B_{p}}=(\overline{B_{p,1}},\overline{B_{p,2}},\cdots,\overline{B_{p,\ell}}) \in \mathcal{G}_{p}$  with $\overline{B_{u,i}}=\epsilon_{u,i}\overline{b_i(x)}+\left<g_u(x)\right>,$ $\overline{B_{v,i}}=\epsilon'_{v,i}\overline{b_i(x)}+\left<g_v^*(x)\right>,$  $\overline{B_{v,i}'}=\epsilon_{v,i}\overline{b_i(x)}+\left<g_v(x)\right>$  and $\overline{B_{p,i}}=\epsilon_{p,i}\overline{b_i(x)}+\left<g_{p}(x)\right>$ for each $i,u,v$ and $p.$

 In view of this,  a $\Lambda$-multi-twisted  code $\mathcal{C}$ of length $n$  over $\mathbb{F}_{q}$ can be uniquely expressed as 
\vspace{-2mm}\begin{equation}\label{a1}\mathcal{C}= \big( \bigoplus\limits_{u=1}^{s} \mathcal{C}_u \big) \oplus \big(\bigoplus\limits_{v=s+1}^{t} \big(\mathcal{C}_v \oplus \mathcal{C}_v' \big)\big) \oplus  \big( \bigoplus\limits_{p=t+1}^{e} \mathcal{C}_{p} \big),\vspace{-2mm}\end{equation}
where $\mathcal{C}_u$ (resp. $\mathcal{C}_v,$ $\mathcal{C}_v'$ and $\mathcal{C}_{p}$) is a subspace of $\mathcal{G}_{u}$ (resp. $\mathcal{G}_v,$  $\mathcal{G}_v'$ and $\mathcal{G}_{p}$) over the field $F_u$ (resp. $F_v,$ $F_v'$ and $F_{p}$) for each $u$ (resp. $v$ and $p$). 
To study their dual codes, we see that if for some $\alpha$ and $i,$ $\epsilon_{\alpha,i}=1,$ then $x^{m_i}=\lambda_i$ in $F_{\alpha},$ which implies that $\lambda_i(x^m-1)/(x^{m_i}-\lambda_i)=m/m_i$ in $F_{\alpha}.$  In view of the above,  the sesquilinear form   corresponding to $\left(\cdot,\cdot\right)$ is a map $\left[ \cdot,\cdot \right]$ from $\left\{(\bigoplus\limits_{u=1}^{s} \mathcal{G}_{u}) \oplus \big( \bigoplus\limits_{v=s+1}^{t}\mathcal{G}_{v}\oplus \mathcal{G}'_{v}\big) \oplus \big(\bigoplus\limits_{p=t+1}^{e} \mathcal{G}_{p}\big) \right\}\times \left\{(\bigoplus\limits_{u=1}^{s} \mathcal{G}_{u}) \oplus \big( \bigoplus\limits_{v=s+1}^{t}\mathcal{H}_{v}\oplus \mathcal{H}'_{v}\big) \oplus \big(\bigoplus\limits_{p=t+1}^{e} \mathcal{G}'_{p}\big)\right\}$ into $ (\bigoplus\limits_{u=1}^{s} F_{u}) \oplus \big( \bigoplus\limits_{v=s+1}^{t}F_{v}\oplus F'_{v}\big) \oplus \big(\bigoplus\limits_{p=t+1}^{e} F_{p}\big),$  which is defined as
   \small{\begin{equation*} \left[A,B\right]=\bigg(\sum_{i=1}^{\ell}\frac{m}{m_i}\epsilon_{1,i} A_{1,i}\overline{B_{1,i}},\sum_{i=1}^{\ell}\frac{m}{m_i}\epsilon_{2,i} A_{2,i}\overline{B_{2,i}},\cdots,\sum_{i=1}^{\ell}\frac{m}{m_i}\epsilon_{s,i} A_{s,i}\overline{B_{s,i}},\sum_{i=1}^{\ell}\frac{m}{m_i}\epsilon_{s+1,i} A_{s+1,i}\overline{B_{s+1,i}'},
\sum_{i=1}^{\ell}\frac{m}{m_i}\epsilon'_{s+1,i} A_{s+1,i}'\overline{B_{s+1,i}}, \end{equation*} \begin{equation}\label{dbform} \cdots, \sum_{i=1}^{\ell}\frac{m}{m_i}\epsilon_{t,i} A_{t,i}\overline{B_{t,i}'},\sum_{i=1}^{\ell} \frac{m}{m_i}\epsilon'_{t,i}A_{t,i}'\overline{B_{t,i}},\sum_{i=1}^{\ell}\frac{m}{m_i}\epsilon_{t+1,i} A_{t+1,i}\overline{B_{t+1,i}},\sum_{i=1}^{\ell}\frac{m}{m_i}\epsilon_{t+2,i} A_{t+2,i}\overline{B_{t+2,i}},\cdots,\sum_{i=1}^{\ell} \frac{m}{m_i}\epsilon_{e,i}A_{e,i}\overline{B_{e,i}}\bigg) \end{equation} }\normalsize for each $A \in V$ and $B \in V'.$  Furthermore, with respect to the sesquilinear form defined by \eqref{dbform},  it is easy to see that the dual code $\mathcal{C}^{\perp}$ of $\mathcal{C}$  is given by  
\begin{equation}\label{decomp}
\mathcal{C}^{\perp} = \big( \bigoplus\limits_{u=1}^{s} \mathcal{C}_u^{\perp} \big) \oplus \big( \bigoplus\limits_{v=s+1}^{t} ( {\mathcal{C}_v'}^{\perp} \oplus \mathcal{C}_v^{\perp}) \big) \oplus  \big( \bigoplus\limits_{p=t+1}^{e} \mathcal{C}_p^{\perp} \big),\vspace{-1mm}\end{equation} 
where $\mathcal{C}_u^{\perp}(\subseteq \mathcal{G}_u)$ is the orthogonal complement of $\mathcal{C}_u$ with respect to $\left[\cdot,\cdot \right]\restriction_{\mathcal{G}_{u} \times \mathcal{G}_{u}}$ for $1 \leq u \leq s;$  $\mathcal{C}_v^{\perp}(\subseteq \mathcal{H}_{v}')$ is the orthogonal complement of $\mathcal{C}_v$ with respect to $\left[\cdot,\cdot \right]\restriction_{\mathcal{H}_{v}' \times \mathcal{G}_{v}},$ $\mathcal{C}_v'^{\perp}(\subseteq \mathcal{H}_{v})$ is the orthogonal complement of  $\mathcal{C}_v^{'}$ with respect to $\left[\cdot,\cdot \right]\restriction_{\mathcal{H}_{v} \times \mathcal{G}_{v}'}$ for $s+1 \leq v \leq t;$ and $\mathcal{C}_{p}^{\perp}(\subseteq \mathcal{G}_{p}')$ is the orthogonal complement of $\mathcal{C}_{p}$ with respect to $\left[\cdot,\cdot\right]\restriction_{\mathcal{G}_{p}'\times \mathcal{G}_{p}}$ for $t+1 \leq {p} \leq e.$ Here $[\cdot,\cdot]\restriction_{\mathcal{G}_u \times \mathcal{G}_u}$ (resp. $[\cdot,\cdot]\restriction_{\mathcal{H}_v' \times \mathcal{G}_v},$  $[\cdot,\cdot]\restriction_{\mathcal{H}_v \times \mathcal{G}_v'})$ and $\left[\cdot,\cdot \right]\restriction_{\mathcal{G}_{p}'\times \mathcal{G}_{p}}$)  is the restriction of the sesquilinear form $[\cdot,\cdot]$ (defined by \eqref{dbform}) to $\mathcal{G}_u \times \mathcal{G}_u$ (resp. $\mathcal{H}_v' \times \mathcal{G}_v,$ $\mathcal{H}_v \times \mathcal{G}_v'$ and $\mathcal{G}_{p}' \times \mathcal{G}_{p}$) for each $u$ (resp. $v$ and $p$).

To count all self-dual and self-orthogonal $\Lambda$-multi-twisted codes, for $s+1 \leq v \leq t,$  let $\mathcal{K}_{v}=\mathcal{G}_{v}\cap \mathcal{H}_{v},$  $\mathcal{K}'_{v}=\mathcal{G}'_{v}\cap \mathcal{H}'_{v},$ and let $\tau_v$ denote the number of integers $i$ satisfying $1 \leq i \leq \ell$ and $\epsilon_{v,i}=\epsilon'_{v,i}.$ Note that $\tau_v=\sum\limits_{i=1}^{\ell}\epsilon_{v,i} \epsilon'_{v,i}$ for each $v.$ One can easily observe that $\text{dim}_{F_v} \mathcal{K}_{v}=\text{dim}_{F'_v} \mathcal{K}'_{v}=\tau_v$ for each $v.$ Then in the following proposition, we characterize all self-dual and self-orthogonal $\Lambda$-multi-twisted codes of length $n$ over $\mathbb{F}_{q}.$

{\prop\label{lem51} Let $\Lambda=(\lambda_1,\lambda_2,\cdots,\lambda_{\ell})$ be fixed. Let  $\mathcal{C}= \big( \bigoplus\limits_{u=1}^{s} \mathcal{C}_u \big) \oplus \big( \bigoplus\limits_{v=s+1}^{t} (\mathcal{C}_v \oplus \mathcal{C}_v' )\big) \oplus \big( \bigoplus\limits_{p=t+1}^{e} \mathcal{C}_{p} \big)$ be a  $\Lambda$-multi-twisted  code of length $n$ over $\mathbb{F}_{q},$ where $\mathcal{C}_u$ (resp. $\mathcal{C}_v,$ $\mathcal{C}_v'$ and $\mathcal{C}_{p}$) is a subspace of $\mathcal{G}_{u}$ (resp. $\mathcal{G}_v,$ $\mathcal{G}_v'$ and $\mathcal{G}_{p}$) over $F_u$ (resp. $F_v,$  $F_v'$ and $F_p$) for each $u$ ($v$ and $p$). Then 
\begin{enumerate}\item[(a)] the code $\mathcal{C}$ is self-dual if and only if all the irreducible factors  of the polynomials $x^{m_1}-\lambda_1,$ $x^{m_2}-\lambda_2, \cdots x^{m_{\ell}}-\lambda_{\ell}$ in $\mathbb{F}_{q}[x]$  are either self-reciprocal or form reciprocal pairs (i.e., $e \leq t$), $\mathcal{C}_u=\mathcal{C}_u^{\perp},$ $\tau_v \geq 1,$ $\mathcal{C}_{v}$ (resp. $\mathcal{C}'_{v}$)  is a subspace of $\mathcal{K}_{v}$ (resp. $\mathcal{K}_{v}'$) satisfying $\mathcal{C}_v'=\mathcal{C}_v^{\perp}\cap \mathcal{K}'_{v}$ for $1 \leq u \leq s$ and $s+1 \leq v \leq t.$ As a consequence, when  all the irreducible factors  of the polynomials $x^{m_1}-\lambda_1,$ $x^{m_2}-\lambda_2, \cdots x^{m_{\ell}}-\lambda_{\ell}$ in $\mathbb{F}_{q}[x]$  are either self-reciprocal or form reciprocal pairs (i.e., $e \leq t$) and $\tau_v \geq 1$ for $s+1 \leq v \leq t,$  the total number of distinct self-dual $\Lambda$-multi-twisted  codes of  length $n$ over $\mathbb{F}_{q}$ is given by 
$\mathfrak{N}_0=\prod\limits_{u=1}^{s} \mathfrak{D}_u \prod\limits_{v=s+1}^{t} \mathfrak{D}_v,$ where $\mathfrak{D}_u$ equals the number of  distinct  $F_u$-subspaces $\mathcal{C}_u$ of $\mathcal{G}_u$ satisfying $\mathcal{C}_u=\mathcal{C}_u^{\perp}$ for  $1\leq u\leq s$ and $\mathfrak{D}_v$ equals the number of distinct $F_v$-subspaces of $\mathcal{K}_v$ for  $s+1 \leq v \leq t.$
\item[(b)] the code $\mathcal{C}$ is self-orthogonal if and only if $\mathcal{C}_u\subseteq \mathcal{C}_u^{\perp},$ $\mathcal{C}_{v}$ (resp. $\mathcal{C}_{v}'$) is a subspace of $\mathcal{K}_{v}$ (resp. $\mathcal{K}'_{v}$) satisfying $\mathcal{C}_v'\subseteq \mathcal{C}_v^{\perp} \cap \mathcal{K}'_{v}$ and $\mathcal{C}_{p}=\{0\}$ for $1 \leq u \leq s,$  $s+1 \leq v \leq t$ and $t+1 \leq {p} \leq e.$ As a consequence, the total number of distinct self-orthogonal $\Lambda$-multi-twisted  codes of length $n$ over $\mathbb{F}_{q}$ is given by $\mathfrak{N}_1=\prod\limits_{u=1}^{s} \mathfrak{E}_u \prod\limits_{v=s+1}^{t} \mathfrak{E}_v,$ where $\mathfrak{E}_u$ equals the number of distinct self-orthogonal $F_u$-subspaces of $\mathcal{G}_u$ for $1 \leq u \leq s$ and $\mathfrak{E}_v$ equals the number of  pairs $(\mathcal{C}_v,\mathcal{C}_v')$ with $\mathcal{C}_v$ (resp. $\mathcal{C}_v'$) as a subspace of $\mathcal{K}_v$ (resp. $\mathcal{K}_v'$) over $F_v$ (resp. $F_v'$)  satisfying $\mathcal{C}_v' \subseteq \mathcal{C}_v^{\perp} \cap \mathcal{K}'_{v}$ for $s+1 \leq v \leq t.$
\end{enumerate}}
\begin{proof} \begin{enumerate}\item[(a)] In view of \eqref{a1} and \eqref{decomp}, we see that the code $\mathcal{C}$ is self-dual if and only if the set $\{g_{t+1}(x), g_{t+2}(x),\cdots,\\g_e(x)\}$ is empty, $\mathcal{C}_u=\mathcal{C}_u^{\perp},$   $\tau_v \geq 1,$ $ \mathcal{C}_v$ is a subspace of $\mathcal{K}_{v}$ and $\mathcal{C}'_{v}$ is a subspace of $\mathcal{K}'_{v}$ satisfying $\mathcal{C}_{v} = \mathcal{C}_v'^{\perp} \cap \mathcal{K}_{v}$ and $ \mathcal{C}_v' = \mathcal{C}_v^{\perp}\cap \mathcal{K}'_{v}$ for each $u$ and $v.$ Further,  for  $s+1 \leq v \leq t,$ if $\tau_v \geq 1,$ $\mathcal{C}_{v}$ is a subspace of $\mathcal{K}_{v}$ and $\mathcal{C}'_{v}$ is a subspace of $\mathcal{K}'_{v},$ then we observe that $\mathcal{C}_{v} = \mathcal{C}_v'^{\perp} \cap \mathcal{K}_{v}$ and $ \mathcal{C}_v' = \mathcal{C}_v^{\perp}\cap \mathcal{K}'_{v}$ hold if and only if $\mathcal{C}'_{v} = \mathcal{C}_v^{\perp} \cap \mathcal{K}'_{v}$ holds. From this,   part (a) follows immediately. 
\item[(b)]  By \eqref{a1} and \eqref{decomp}, we see that the code $\mathcal{C}$ is self-orthogonal if and only if  $\mathcal{C}_u \subseteq \mathcal{C}_u^{\perp},$  $\mathcal{C}_{v}$ (resp. $\mathcal{C}_{v}'$) is a subspace of $\mathcal{K}_{v}$ (resp. $\mathcal{K}'_{v}$) satisfying $\mathcal{C}_v'\subseteq \mathcal{C}_v^{\perp} \cap \mathcal{K}'_{v}$ and    $\mathcal{C}_{v} \subseteq \mathcal{C}_{v}'^{\perp} \cap \mathcal{K}_{v},$ and $\mathcal{C}_{p} \subseteq \{0\},$ $\{0\} \subseteq \mathcal{C}_{p}^{\perp}$ for each $u,v$ and $p.$ Further, for $s+1 \leq v \leq t,$ we see that if $\mathcal{C}_{v}$ (resp. $\mathcal{C}_{v}'$) is a subspace of $\mathcal{K}_{v}$ (resp. $\mathcal{K}'_{v}$), then  $\mathcal{C}_v'\subseteq \mathcal{C}_v^{\perp} \cap \mathcal{K}'_{v}$ and    $\mathcal{C}_{v} \subseteq \mathcal{C}_{v}'^{\perp} \cap \mathcal{K}_{v}$ hold if and only if $\mathcal{C}_v'\subseteq \mathcal{C}_v^{\perp} \cap \mathcal{K}'_{v}$ holds. From this, part (b) follows.
\end{enumerate}
\vspace{-5mm}\end{proof}
Next let  us define $\mathcal{I}_1=\{u: 1 \leq u \leq s, d_u=1\}$ and $ \mathcal{I}_2=\{u: 1 \leq u \leq s, d_u >1 \}.$ We make the following observation.
{\lem\label{52} For $1 \leq u \leq s,$ let $\left[\cdot, \cdot\right]\restriction_{\mathcal{G}_{u} \times \mathcal{G}_{u}}$ denote the restriction of the sesquilinear form $\left[\cdot,\cdot\right]$ to $\mathcal{G}_{u}\times \mathcal{G}_{u}.$ Then the following hold.
\begin{enumerate}\item[(a)] When $u \in \mathcal{I}_{1},$ $\left[\cdot, \cdot\right]\restriction_{\mathcal{G}_{u} \times \mathcal{G}_{u}}$ is a non-degenerate, reflexive and symmetric bilinear form on $\mathcal{G}_{u},$ i.e., \\ $\left(\mathcal{G}_{u},  \left[\cdot, \cdot\right]\restriction_{\mathcal{G}_{u} \times \mathcal{G}_{u}}\right)$ is an orthogonal space of dimension $\epsilon_u$ over $F_u \simeq \mathbb{F}_{q}.$ \item[(b)] When $u \in \mathcal{I}_{2},$ $\left[\cdot, \cdot\right]\restriction_{\mathcal{G}_{u} \times \mathcal{G}_{u}}$ is a non-degenerate, reflexive and Hermitian  form on $\mathcal{G}_{u},$ i.e., $\left(\mathcal{G}_{u},  \left[\cdot, \cdot\right]\restriction_{\mathcal{G}_{u} \times \mathcal{G}_{u}}\right)$ is a unitary space of dimension $\epsilon_u$ over $F_u \simeq \mathbb{F}_{q^{d_u}}.$  \end{enumerate} }
\begin{proof} Proof is trivial.
\end{proof}
Now we proceed to count all self-dual and self-orthogonal $\Lambda$-multi-twisted codes of length $n$ over $\mathbb{F}_{q}.$
\subsection{Enumeration formula for self-dual $\Lambda$-multi-twisted codes} 
 In the following theorem, we derive necessary and sufficient conditions for the existence of a self-dual $\Lambda$-multi-twisted code of length $n$  over $\mathbb{F}_{q}.$
 We also provide enumeration formula for this special class of multi-twisted codes.
\noindent\begin{thm}\label{enum1}Let $\Lambda=(\lambda_1,\lambda_2,\cdots,\lambda_{\ell})$ be fixed. For $s+1 \leq v \leq t,$   let $\tau_v$ denote the number of integers $i$ satisfying $1 \leq i \leq \ell$ and $\epsilon_{v,i}=\epsilon'_{v,i}.$ \begin{enumerate}  \item[(a)] There exists a self-dual  $\Lambda$-multi-twisted  code of  length $n$ over $\mathbb{F}_{q}$ if and only if  all the irreducible factors  of the polynomials $x^{m_1}-\lambda_1,$ $x^{m_2}-\lambda_2, \cdots x^{m_{\ell}}-\lambda_{\ell}$ in $\mathbb{F}_{q}[x]$  are either self-reciprocal or form reciprocal pairs (i.e., $e \leq t$), $\tau_v \geq 1$ for $s+ 1\leq v \leq t,$ $\epsilon_u$ is even for $1 \leq u \leq s$ and $(-1)^{\epsilon_u/2}$ is a square in $\mathbb{F}_{q}$ for all $u \in \mathcal{I}_{1}.$ 
 \item[(b)] When all the irreducible factors  of the polynomials $x^{m_1}-\lambda_1,$ $x^{m_2}-\lambda_2, \cdots x^{m_{\ell}}-\lambda_{\ell}$ in $\mathbb{F}_{q}[x]$  are either self-reciprocal or form reciprocal pairs (i.e., $e \leq t$), $\tau_v \geq 1$ for $s+ 1\leq v \leq t,$ $\epsilon_u$ is even for $1 \leq u \leq s$ and $(-1)^{\epsilon_u/2}$ is a square in $\mathbb{F}_{q}$ for all $u \in \mathcal{I}_{1},$ the number $\mathfrak{N}_{0}$ of distinct self-dual $\Lambda$-multi-twisted  codes of length $n$ over $\mathbb{F}_{q}$ is given by
$\mathfrak{N}_{0}=\prod\limits_{u=1}^s \mathfrak{D}_u \prod\limits_{v=s+1}^t \left( \sum\limits_{b=0}^{\tau_v} {\tau_v \brack b}_{q^{d_v}}\right),$ where  
$$\mathfrak{D}_u= \left\{ \begin{array}{ll}
\prod\limits_{a=0}^{\epsilon_u/2 -1} \big(q^a+1\big) & \text{if }u\in \mathcal{I}_1 \text{ \&\ }q \text{ is odd;} \\ 
\prod\limits_{a=1}^{\epsilon_u/2 -1} \big( q^{a}+1\big) & \text{if }u\in \mathcal{I}_1 \text{ \& }q \text{ is even;}\\
\prod\limits_{a=0}^{\epsilon_u/2-1} \big( q^{(2a+1)d_u/2} +1 \big) & \text{if }u \in\mathcal{I}_2. 
 \end{array}\right. $$
\end{enumerate}
\end{thm}
 In order to prove this theorem, we need to prove the following lemma:
{\lem\label{lem53} Let $1 \leq u \leq s$ be fixed. There exists an $F_u$-subspace $\mathcal{C}_u$ of $\mathcal{G}_u$ satisfying $\mathcal{C}_u=\mathcal{C}_u^{\perp}$ if and only if the following two conditions are satisfied: (i) $\epsilon_u$ is an even integer, and (ii)
 $(-1)^{\epsilon_u / 2}$ is a square in $\Fq$ for all $u \in \mathcal{I}_{1}.$ (Here $\mathcal{C}_u^{\perp}(\subseteq \mathcal{G}_{u})$ is the orthogonal complement of $\mathcal{C}_u$ with respect to $\left[\cdot, \cdot\right]\restriction_{\mathcal{G}_{u}\times \mathcal{G}_{u}}.$)
 \begin{proof} To prove this, we see, by Lemma \ref{52}, that $\left(\mathcal{G}_{u}, \left[\cdot, \cdot\right]\restriction_{\mathcal{G}_{u} \times \mathcal{G}_{u}} \right)$ is an orthogonal space having dimension $\epsilon_u$ over $F_u$ when $u \in \mathcal{I}_{1}$ and that $\left(\mathcal{G}_{u}, \left[\cdot, \cdot\right]\restriction_{\mathcal{G}_{u} \times \mathcal{G}_{u}} \right)$ is a unitary space having dimension $\epsilon_u$ over $F_u$ when $u \in \mathcal{I}_{2}.$ Now if $\mathcal{C}_{u}$ is an $F_u$-subspace of $\mathcal{G}_{u},$ then by Theorem \ref{dimgen}, we see that $\text{dim}_{F_u} \mathcal{C}_{u}^{\perp}=\epsilon_u-\text{dim}_{F_u}\mathcal{C}_{u}.$ Further, if $\mathcal{C}_u$  satisfies $\mathcal{C}_u=\mathcal{C}_u^{\perp},$ then we get $\epsilon_u=2\text{ dim}_{F_u}\mathcal{C}_u,$ which implies that $\epsilon_u$ is an even integer. 

On the other hand, when $u \in \mathcal{I}_{2}$ and $\epsilon_u$ is even, by  Theorem \ref{unitary}(a), we see that the Witt index of $\left(\mathcal{G}_{u}, \left[\cdot, \cdot\right]\restriction_{\mathcal{G}_{u} \times \mathcal{G}_{u}} \right)$ is $\epsilon_u/2,$ so    there exists an $F_u$-subspace $\mathcal{C}_u$ of $\mathcal{G}_u$ satisfying $\mathcal{C}_u=\mathcal{C}_u^{\perp}.$ When $u \in \mathcal{I}_{1}$ and $\epsilon_u$ is even, by Theorem \ref{9}(a), we see that the Witt index of $\left(\mathcal{G}_{u}, \left[\cdot, \cdot\right]\restriction_{\mathcal{G}_{u} \times \mathcal{G}_{u}} \right)$ is $\epsilon_u/2$ if and only if $(-1)^{{\epsilon_u}/{2}}$ is a square in $\mathbb{F}_{q}.$ That is, when $u \in \mathcal{I}_{1}$ and $\epsilon_u$ is even,  there exists an $F_u$-subspace $\mathcal{C}_u$ of $\mathcal{G}_u$ satisfying $\mathcal{C}_u=\mathcal{C}_u^{\perp}$ if and only if $(-1)^{{\epsilon_u}/{2}}$ is a square in $\mathbb{F}_{q}.$ This proves the lemma. 
\end{proof}
}
\noindent\textit{Proof of Theorem \ref{enum1}.} Part (a) follows immediately by  Proposition \ref{lem51}(a) and Lemma \ref{lem53}. To prove (b),  we see, by Proposition \ref{lem51}(a) again, that it is enough to determine the numbers $\mathfrak{D}_{u}$ for all $u \in \mathcal{I}_{1}\cup \mathcal{I}_{2}$ and $\mathfrak{D}_{v}$ for $s+1 \leq v \leq t.$

To do this, we see, by \eqref{sub},  that for each $v~(s+1 \leq v \leq t),$ the number $\mathfrak{D}_v$ of distinct $F_v$-subspaces of $\mathcal{K}_{v}$ equals $\mathfrak{D}_v=N(\tau_v,q^{d_v})=\sum\limits_{b=0}^{\tau_v} {{\tau_v \brack b}_{q^{d_v}}}.$  Moreover, for $u \in \mathcal{I}_1,$ by Lemma \ref{52}(a) and Theorem \ref{9}(b),  we  see that the number $\mathfrak{D}_u$ of distinct  $\epsilon_u/2$-dimensional self-orthogonal subspaces of $\mathcal{G}_u$ over $F_u$  is given by $\mathfrak{D}_u=\prod\limits_{a=0}^{\epsilon_u /2 -1}\big( q^a +1 \big)$ when $q$ is odd, while the number $\mathfrak{D}_u$ of such subspaces is given by $\mathfrak{D}_u=\prod\limits_{a=1}^{\epsilon_u /2 -1}\big( q^a +1 \big)$ when $q$ is even. For $u \in \mathcal{I}_2,$ by Lemma \ref{52}(b) and Theorem \ref{unitary}(b),   we see that the number $\mathfrak{D}_u$ of distinct  $\epsilon_u/2$-dimensional  self-orthogonal subspaces of $\mathcal{G}_u$ over $F_u$ is given by $\mathfrak{D}_u= \prod\limits_{a=0}^{\epsilon_u/2-1} \big( q^{(2a+1)d_u/2} +1 \big).$ From this and using Proposition \ref{lem51}(a) again, part (b) follows immediately. $\hfill \Box$
\subsection{Enumeration formula for self-orthogonal $\Lambda$-multi-twisted codes}
In the following theorem, we enumerate all self-orthogonal $\Lambda$-multi-twisted  codes of length $n$  over $\mathbb{F}_{q}.$ 
{\thm \label{countselfortho} Let $\Lambda=(\lambda_1,\lambda_2,\cdots,\lambda_{\ell})$ be fixed. For $s+1 \leq v \leq t,$   let $\tau_v$ denote the number of integers $i$ satisfying $1 \leq i \leq \ell$ and $\epsilon_{v,i}=\epsilon'_{v,i}.$ 
 The number $\mathfrak{N}_{1}$  of distinct self-orthogonal $\Lambda$-multi-twisted  codes  of length $n$ over $\mathbb{F}_{q}$ is given by
 $$\mathfrak{N}_{1}=\prod\limits_{u =1}^s  \mathfrak{E}_{u}    
\prod\limits_{v=s+1}^t \left( \sum\limits_{k_1=0}^{\tau_v} {\tau_v \brack k_1}_{q^{d_v}} \Big(\sum\limits_{k_2=0}^{\tau_v - k_1} {\tau_v-k_1 \brack k_2}_{q^{d_v}}\Big)\right),$$ where
$$\mathfrak{E}_{u}=\left\{ \begin{array}{ll}
\sum\limits_{k=0}^{\epsilon_{u}/{2}} \left( {\epsilon_u/{2} \brack k}_q \prod\limits_{a=0}^{k-1} \big(q^{\epsilon_u/{2} -a-1} +1\big) \right)  & \text{if }u \in \mathcal{I}_1  \text{ \& either } q \equiv 1(\text{mod 4), } \epsilon_u \text{ is even or } \\  &   \epsilon_u \equiv 0(\text{mod 4), } q \equiv 3~(\text{mod }4); \\

\sum\limits_{k=0}^{(\epsilon_u-2)/{2}} \left( {(\epsilon_u-2)/{2} \brack k}_q \prod\limits_{a=0}^{k-1} \big(q^{\epsilon_u/{2} -a} +1\big) \right) & \text{if } u \in \mathcal{I}_1,~ q \equiv 3(\text{mod 4) \& } \epsilon_u \equiv 2~(\text{mod }4); \\

\sum\limits_{k=0}^{{(\epsilon_u-1)}/{2}} \left( {(\epsilon_u-1)/{2} \brack k}_q \prod\limits_{a=0}^{k-1} \big(q^{{(\epsilon_u-1)}/{2} -a} +1\big) \right) & \text{if } u \in \mathcal{I}_1 \text{ \&  both } q,~\epsilon_u \text{ are odd;} \\

\sum\limits_{k=0}^{(\epsilon_u-1)/{2}} \left( {(\epsilon_u-1)/{2} \brack k}_q \prod\limits_{a=0}^{k-1} \big(q^{(\epsilon_u -2a -1)/{2}}+1\big)\right) & \text{if } u \in \mathcal{I}_1,~q \text{ is even \& }\epsilon_u\text{ is odd;}\\

\sum\limits_{k=0}^{(\epsilon_u-2)/{2}} {(\epsilon_u-2)/{2} \brack k}_q \prod\limits_{a=0}^{k-1}\big(q^{(\epsilon_u-2a-2)/{2}}+1\big)  & \\

 ~~~~ +\sum\limits_{k'=1}^{\epsilon_u/{2}} q^{\epsilon_u-2k'} {(\epsilon_u-2)/{2} \brack k'-1}_q \prod\limits_{a'=0}^{k'-2}\big(q^{({\epsilon_u-2a'-2})/{2}}+1\big) & \text{if }u \in \mathcal{I}_1 \text{ \& both }q,~ \epsilon_u\text{ are even;} \\
 
  \sum\limits_{k=0}^{\epsilon_u/2} \big(\prod\limits_{a=\epsilon_u+1-2k}^{k}(q^{\frac{ad_u}{2}}-(-1)^a)\big)/\big(\prod\limits_{j=1}^{k}(q^{jd_u}-1)\big)  & \text{if }u \in \mathcal{I}_2 \text{ \& }\epsilon_u \text{ is even;}\\

  \sum\limits_{k=0}^{(\epsilon_u-1)/2} \big(\prod\limits_{a=\epsilon_u+1-2k}^{k}(q^{\frac{ad_u}{2}}-(-1)^a)\big)/\big(\prod\limits_{j=1}^{k}(q^{jd_u}-1)\big)   & \text{if }u \in \mathcal{I}_2 \text{ \& }\epsilon_u \text{ is odd.}
  \end{array} \right.$$}
\begin{proof} By Proposition \ref{lem51}(b),  we see that  to determine the number $\mathfrak{N}_{1},$ it is enough to determine the numbers $\mathfrak{E}_{u}$ for all $u \in \mathcal{I}_{1}\cup \mathcal{I}_{2}$ and $\mathfrak{E}_{v}$ for $s+1 \leq v \leq t.$
\\\textbf{I.}  First let $u \in \mathcal{I}_{1}.$ Here we  see, by  Lemma \ref{52}(a),  that for  $u \in \mathcal{I}_1,$ $(\mathcal{G}_u,[\cdot,\cdot]\restriction_{\mathcal{G}_u\times\mathcal{G}_u})$ is an  $\epsilon_u$-dimensional orthogonal space  over $ F_u \simeq \mathbb{F}_{q}.$  Now we shall distinguish the following two cases: (i) $q$ is odd and (ii)  $q$ is even. 
\\ \textbf{(i)} When $q$ is  odd, one can view $\mathcal{G}_u$ as a non-degenerate quadratic space over $F_u$ with respect to the quadratic map $Q_u : \mathcal{G}_{u} \rightarrow F_u,$ defined as $Q_u({a}(x))=\frac{1}{2}\left[{a}(x),{a}(x)\right]$ for all ${a}(x) \in \mathcal{G}_{u}.$ In view of this, we see, by Theorem \ref{quadratic}(a),  that the Witt index of $\mathcal{G}_{u}$ is given by \begin{equation}\label{e3.6}w_u=\left\{ \begin{array}{ll}
\epsilon_u/2 & \text{if either }\epsilon_u \text{ is even and }q \equiv 1~(\text{mod }4) \text{ or }  \epsilon_u \equiv 0~(\text{mod }4) \text{ and }q \equiv 3~(\text{mod }4) ;  \\

(\epsilon_u-2)/2 & \text{if }  \epsilon_u \equiv 2~(\text{mod }4) \text{ and }q \equiv 3~(\text{mod }4);   \\

(\epsilon_u-1)/2 & \text{if } \epsilon_u \text{ is odd.}  \end{array}\right.\end{equation}   
Further,  by Theorem \ref{quadratic}(b),   we see that  the number $\mathfrak{E}_u$ of distinct self-orthogonal subspaces of $\mathcal{G}_u$ over $F_u$ is given by  $\mathfrak{E}_u=\sum\limits_{k=0}^{w_u} \left( {w_u \brack k}_q \prod\limits_{a=0}^{k-1} \big(q^{w_u-\epsilon-a} +1\big) \right),$ where $w_u$ (the Witt index of $\mathcal{G}_u$) is given by \eqref{e3.6} and $\epsilon=1$ if $w_u=\epsilon_u/2,$ $\epsilon=-1$ if $w_u=(\epsilon_u-2)/2,$ while $\epsilon=0$ if $w_u=(\epsilon_u-1)/2.$
\\\textbf{(ii)} Next let $q$ be even.  Let us define $\mathcal{V}_u=\{(\epsilon_{u,1}c_{u,1},\epsilon_{u,2}c_{u,2},\cdots,\epsilon_{u,\ell}c_{u,\ell}) \in \mathcal{G}_u: \sum\limits_{i=1}^{\ell} \epsilon_{u,i}c_{u,i}=0\}.$ Note that $\mathcal{V}_{u}$ is an $F_u$-subspace of $\mathcal{G}_u$ and $\text{dim}_{F_u} \mathcal{V}_u = \epsilon_u -1.$ Let $\theta_u= \left( \epsilon_{u,1},\epsilon_{u,2},\cdots,\epsilon_{u,\ell} \right) \in \mathcal{G}_u.$ Since $\sum\limits_{i=1}^{\ell} \epsilon_{u,i}= \epsilon_u,$ we see that $\theta_u \in \mathcal{V}_{u}$ if and only if $\epsilon_u$ is even. 

When $\epsilon_u$ is odd, we see that $\theta_u \notin \mathcal{V}_u,$ which implies that $\mathcal{G}_u = \mathcal{V}_u \oplus \left<\theta_u \right>.$ Next it is easy to observe that any self-orthogonal $F_u$-subspace of $\mathcal{G}_u$ is contained in $\mathcal{V}_u$ and that  $[c,\theta_u]=0$ for each $c\in \mathcal{V}_u.$ Further, we note that as $q$ is even, all $m_i$'s are odd. This implies that $m$ is odd, which further implies that $\frac{m}{m_i}=1$ in $F_u.$ Moreover, as $u \in \mathcal{I}_{1},$ the conjugation $^{-}$ is the identity map on $F_{u}.$ This implies that for each $c_u=(\epsilon_{u,1}c_{u,1},\epsilon_{u,2}c_{u,2},\cdots,\epsilon_{u,\ell}c_{u,\ell}) \in \mathcal{V}_{u},$ we have $[c_u,c_u]=\sum\limits_{i=1}^{\ell}\epsilon_{u,i}c_{u,i}^2\frac{m}{m_i}=(\sum\limits_{i=1}^{\ell}\epsilon_{u,i}c_{u,i})^2=0.$ From this, it follows that $[\cdot,\cdot]\restriction_{\mathcal{V}_u\times\mathcal{V}_u}$ is a non-degenerate, reflexive and alternating bilinear form on $\mathcal{V}_u,$ i.e., 
 $(\mathcal{V}_u,[\cdot,\cdot]\restriction_{\mathcal{V}_u\times\mathcal{V}_u})$ is a symplectic space over $F_u$ having the dimension as $\epsilon_u-1$  and the Witt index as $\frac{\epsilon_u-1}{2}.$ Now by Theorem \ref{symplectic}(b),  we see that for $0\leq k \leq \frac{\epsilon_u-1}{2},$ the number of distinct $k$-dimensional self-orthogonal subspaces of $\mathcal{V}_u$ (and hence of $\mathcal{G}_{u}$) is given by $ {(\epsilon_u-1)/2 \brack k}_q \prod\limits_{a=0}^{k-1} \big(q^{\frac{\epsilon_u -2a -1}{2}}+1\big).$ This implies that  the number $\mathfrak{E}_u$ of distinct self-orthogonal subspaces of $\mathcal{G}_u$ over $F_u$ is given by
$\mathfrak{E}_u=\sum\limits_{k=0}^{\frac{\epsilon_u-1}{2}} \left( {(\epsilon_u-1)/2 \brack k}_q \prod\limits_{a=0}^{k-1} \big(q^{\frac{\epsilon_u -2a -1}{2}}+1\big)\right).$ 

On the other hand, when  $\epsilon_u$ is even, we see that $\theta_u \in \mathcal{V}_u.$ Let $\widehat{\mathcal{V}}_u$ be an $(\epsilon_u-2)$-dimensional $F_u$-subspace of $\mathcal{V}_u$ such that $\theta_u \notin \widehat{\mathcal{V}}_u.$ Then we have $\mathcal{V}_u= \widehat{\mathcal{V}}_u \oplus \left< \theta_u \right>.$ Next we observe that there exists $z_u \in \widehat{\mathcal{V}}_u^{\perp} \setminus \mathcal{V}_u.$ From this, it follows that $\mathcal{G}_u= \widehat{\mathcal{V}}_u \oplus \left<z_u\right> \oplus \left<\theta_u\right>.$ It is easy to see that any self-orthogonal $F_u$-subspace of $\mathcal{G}_u$ is contained in $\mathcal{V}_u=\widehat{\mathcal{V}}_u \oplus \left< \theta_u\right>,$ which  implies that any self-orthogonal subspace of $\mathcal{G}_u$ is either (i) contained in $\widehat{\mathcal{V}}_u,$ or (ii) contained in $\widehat{\mathcal{V}}_u \oplus \left< \theta_u\right>$ but not in $\widehat{\mathcal{V}}_u.$ Further, we observe that $(\widehat{\mathcal{V}}_u, [\cdot,\cdot]\restriction_{\widehat{\mathcal{V}}_u\times \widehat{\mathcal{V}}_u})$ is a symplectic space over $F_u$ having the dimension as $\epsilon_u-2$ and the Witt index as $(\epsilon_u-2)/2.$ Now by Theorem \ref{symplectic}(b),  we see that for $0 \leq k \leq \frac{\epsilon_u-2}{2},$ the number $\mathfrak{E}_u$ of distinct $k$-dimensional totally isotropic  subspaces of $\mathcal{G}_u$ is given by $\mathfrak{E}_u={(\epsilon_u-2)/{2} \brack k}_q \prod\limits_{a=0}^{k-1}\big(q^{\frac{\epsilon_u-2a-2}{2}}+1\big).$ 
Next we proceed to count all $k$-dimensional $F_u$-subspaces that are contained in $\widehat{\mathcal{V}}_u \oplus \left<\theta_u\right>$ but not in $\widehat{\mathcal{V}}_u.$ To do this, we observe that for $1 \leq k \leq \epsilon_u/2,$ any such $k$-dimensional  $F_u$-subspace of $\mathcal{G}_u$  is of the type $\left<y_1,y_2,\cdots,y_{k-1},\theta_u+y_k\right>,$ where $y_r \in \widehat{\mathcal{V}}_u \setminus \{0\}$ for $1\leq r \leq k-1$ and $y_k \in \widehat{\mathcal{V}}_u.$ 
We further observe that the $k$-dimensional $F_u$-subspace $\left<y_1,y_2,\cdots,y_{k-1},\theta_u+y_k\right>$ of $\mathcal{G}_u$ is self-orthogonal if and only if $\left<y_1,y_2,\cdots,y_{k-1}\right> $ is a self-orthogonal $F_u$-subspace of $\widehat{\mathcal{V}}_u$ and $y_k \in \left<y_1,y_2,\cdots,y_{k-1}\right>^{\perp}.$ Now by Theorem \ref{symplectic}(b),   for $1 \leq k \leq \epsilon_u/2,$ we see that the number of distinct $(k-1)$-dimensional self-orthogonal $F_u$-subspaces of $\widehat{\mathcal{V}}_u$ is given by ${(\epsilon_u-2)/2 \brack k-1}_q \prod\limits_{a=0}^{k-2} \big( q^{\frac{\epsilon_u-2a-2}{2}}+1\big).$ Next we observe that for $y_k,y_k' \in \left<y_1,y_2,\cdots,y_{k-1}\right>^{\perp} \setminus \left<y_1,y_2,\cdots,y_{k-1}\right>, \left< y_1,y_2\cdots,y_{k-1},\theta_u+y_k\right> = \left< y_1,y_2\cdots,y_{k-1},\theta_u+y_k'\right>$ if and only if $y_k-y_k' \in \left<y_1,y_2\cdots,y_{k-1}\right>,$ i.e., all $y_k$'s lying in different cosets of $\left< y_1,y_2\cdots,y_{k-1}\right>^{\perp}/ \left<y_1,y_2,\cdots,y_{k-1}\right>$ give rise to distinct self-orthogonal spaces of the type $\left<y_1,y_2,\cdots,y_{k-1},\theta_u+y_k\right>.$ 
We also observe that the $F_u$-dimension of $\left<y_1,y_2,\cdots,y_{k-1}\right>^{\perp}$ is $\epsilon_u-2-(k-1),$ which implies that $y_k$ has  $q^{\epsilon_u-2k}$ relevant choices. Therefore for $1 \leq k \leq \epsilon_u/2,$ the number of distinct $k$-dimensional $F_u$-subspaces of $\mathcal{G}_u$ that are contained in $\widehat{\mathcal{V}}_u \oplus \left<\theta_u\right>$ but not in $\widehat{\mathcal{V}}_u,$ is given by $ q^{\epsilon_u-2k} {({\epsilon_u-2})/{2} \brack k-1}_q \prod\limits_{a=0}^{k-2}\big(q^{\frac{\epsilon_u-2a-2}{2}}+1\big).$ On combining both the cases, we see that the number $\mathfrak{E}_u$ of distinct self-orthogonal $F_u$-subspaces of $\mathcal{G}_u$ is given by $\mathfrak{E}_u=\sum\limits_{k=0}^{\frac{\epsilon_u-2}{2}} {({\epsilon_u-2})/{2} \brack k}_q \prod\limits_{a=0}^{k-1}\big(q^{\frac{\epsilon_u-2a-2}{2}}+1\big)+ \sum\limits_{k'=1}^{\epsilon_u/2} q^{\epsilon_u-2k'} {({\epsilon_u-2})/{2} \brack k'-1}_q \prod\limits_{a'=0}^{k'-2}\big(q^{\frac{\epsilon_u-2a'-2}{2}}+1\big)$ when $\epsilon_u$ is even.
\\\textbf{II.} Next let $u \in \mathcal{I}_2.$ Here we observe, from Lemma \ref{52} (b), that  $(\mathcal{G}_u,[\cdot,\cdot]\restriction_{\mathcal{G}_u\times\mathcal{G}_u})$ is a unitary space over $F_u$ having dimension $\epsilon_u.$ Further, by Theorem \ref{unitary}(a), the Witt index $w_u$ of $\mathcal{G}_u$ is given by 
$$w_u=\left\{\begin{array}{ll}\epsilon_u/2 & \text{if }\epsilon_u \text{ is even;}\\(\epsilon_u-1)/2 & \text{if }\epsilon_u \text{ is odd.}\end{array}\right.$$ Now by Theorem \ref{unitary}(b),    we see that the number $\mathfrak{E}_u$ of distinct self-orthogonal $F_u$-subspaces of $\mathcal{G}_u$ is given by 
$\mathfrak{E}_u=
 \sum\limits_{k=0}^{w_u} \big(\prod\limits_{a=\epsilon_u+1-2k}^{k}(q^{\frac{ad_u}{2}}-(-1)^a)\big)/\big(\prod\limits_{j=1}^{k}(q^{jd_u}-1)\big).$
\\\textbf{III.} Finally,  for $s+1 \leq v \leq t,$ we shall  count the number of pairs $(\mathcal{C}_v,\mathcal{C}_v')$ with $\mathcal{C}_v$ as an $F_v$-subspace of $\mathcal{K}_{v}$ and $\mathcal{C}_v'$ as an $F_v'$-subspace of $\mathcal{K}_{v}'$ satisfying $\mathcal{C}_v' \subseteq \mathcal{C}_v^{\perp} \cap \mathcal{K}'_{v}.$ In order to do this, we note that $(\mathcal{K}_v \times \mathcal{K}'_{v},[\cdot,\cdot]\restriction_{\mathcal{K}_v\times\mathcal{K}'_v})$ is non-degenerate. So if the dimension of $\mathcal{C}_v$ is $k_1,$ then one can observe that the dimension of $\mathcal{C}_v^{\perp} \cap \mathcal{K}'_{v}$ is $\tau_v-k_1,$ where $0 \leq k_1 \leq \tau_v.$ As $\mathcal{C}_v'$ has to be a subspace of $\mathcal{C}_v^{\perp} \cap \mathcal{K}'_{v},$ by \eqref{sub}, $\mathcal{C}_v'$ has $N(\tau_v-k_1,q^{d_v})=\sum\limits_{k_2=0}^{\tau_v - k_1} {\tau_v-k_1 \brack k_2}_{q^{d_v}}$ choices if  $\text{dim}_{F_v}\mathcal{C}_v=k_1.$  Further, we see that the number of distinct $k_1$-dimensional $F_v$-subspaces of $\mathcal{G}_{v}$ is given by ${\tau_v \brack k_1}_{q^{d_v}}.$ From this, it follows the number $\mathfrak{E}_v$ of pairs $(\mathcal{C}_v,\mathcal{C}_v')$ with $\mathcal{C}_v$ as an $F_v$-subspace of $\mathcal{K}_{v}$ and $\mathcal{C}_v'$ as an $F_v'$-subspace of $\mathcal{K}_{v}'$ satisfying $\mathcal{C}_v' \subseteq \mathcal{C}_v^{\perp} \cap \mathcal{K}_{v}'$ is given by $\mathfrak{E}_v= \sum\limits_{k_1=0}^{\tau_v} {\tau_v \brack k_1}_{q^{d_v}} \big(\sum\limits_{k_2=0}^{\tau_v - k_1} {\tau_v-k_1 \brack k_2}_{q^{d_v}}\big).$ Now using Proposition \ref{lem51}(b) again,  the desired result follows immediately.\end{proof}
\section{Some more results on multi-twisted codes over finite fields}\label{Gen}
A $\Lambda$-multi-twisted $\mathcal{C}$ of length $n$ over $\mathbb{F}_{q}$ is called a $\rho$-generator code if $\rho$ is the smallest positive integer with the property that there exist $\rho$ number of codewords ${a_1}(x),{a_2}(x),\cdots,{a_{\rho}}(x)\in \mathcal{C}$ such that every ${c}(x)\in \mathcal{C}$ can be expressed as ${c}(x)=f_1(x){a}_1(x)+f_2(x){a}_{2}(x)+\cdots+f_{\rho}(x){a}_{\rho}(x)$ for some $f_1(x),f_2(x),\cdots, f_{\rho}(x) \in \mathbb{F}_{q}[x],$ and we denote $\mathcal{C} = \big< {a_1}(x),{a_2}(x),\cdots,{a_{\rho}}(x) \big>.$  In a recent work, Aydin and  Halilovic \cite{aydin} studied some basic properties of 1-generator $\Lambda$-multi-twisted codes of  length $n$ over $\mathbb{F}_{q}.$ In this section, we shall study some basic properties of $\rho$-generator $\Lambda$-multi-twisted codes over finite fields.

 Let $\mathcal{C}= \big< {a_1}(x),{a_2}(x),\cdots,{a_{\rho}}(x) \big>$ be a $\rho$-generator $\Lambda$-multi-twisted  code  of length $n$ over $\mathbb{F}_{q},$ where ${a_\kappa}(x)=(a_{\kappa,1}(x),a_{\kappa,2}(x),\cdots,a_{\kappa,\ell}(x))$ for  $1 \leq \kappa \leq \rho.$ For $1 \leq i \leq \ell,$ if $\pi_i$ is the projection of $V$ onto $\mathbb{F}_{q}[x]/\langle x^{m_i}-\lambda_i \rangle,$ then it is easy to observe that $\pi_i(\mathcal{C})$ is a $\lambda_i$-constacyclic code of length $m_i$ over $\mathbb{F}_{q}$ with the generator polynomial as  $\gcd(a_{1,i}(x),a_{2,i}(x),\cdots,a_{\rho,i}(x), x^{m_i} - \lambda_i).$ Further, the annihilator of $\mathcal{C}$ is defined as $\text{Ann}(\mathcal{C})=\{ f(x)\in \Fq [x]: f(x){a_\kappa}(x) ={0} \text{ in } V \text{ for } 1 \le \kappa \le \rho \}.$ It  is easy to see that $\text{Ann}(\mathcal{C})$ is an ideal of the principal ideal ring $\mathbb{F}_{q}[x].$ Note that $\prod\limits_{i=1}^{\ell}(x^{m_i}-\lambda_i) \in \text{Ann}(\mathcal{C}).$ Therefore  there exists a unique smallest degree monic polynomial $h(x)\in\mathbb{F}_{q}[x],$ which generates $\text{Ann}(\mathcal{C});$  the polynomial $h(x)$ is called the parity-check polynomial of  $\mathcal{C}.$ In the following theorem, we determine the parity-check polynomial of a $\rho$-generator $\Lambda$-multi-twisted code of length $n$ over $\mathbb{F}_{q}.$
{\thm\label{parity0} Let $\mathcal{C}=\big< {a_1}(x),{a_2}(x),\cdots,{a_{\rho}}(x) \big>$ be a $\rho$-generator $\Lambda$-multi-twisted  code of  length $n$ over $\mathbb{F}_{q},$ where ${a_\kappa}(x)=(a_{\kappa,1}(x),a_{\kappa,2}(x),\cdots,a_{\kappa,\ell}(x))$ for  $1 \le \kappa \le \rho.$ Let $w_i(x)=\gcd(a_{1,i}(x),a_{2,i}(x),\cdots,a_{\rho,i}(x), x^{m_i} - \lambda_i)$ for $1 \leq i \leq \ell.$
 Then the following hold.
 \begin{enumerate}\item[(a)]  The parity-check polynomial $h(x)$ of $\mathcal{C}$ is given by $h(x)= \underset{1 \le i \le \ell}{\text{lcm}} \left[ \frac{ {x^{m_i} - \lambda_i}}{{w_i(x)}}\right].$ \item[(b)] When $\rho=1,$ we have $\text{dim}_{\mathbb{F}_{q}} \mathcal{C}=\text{deg }h(x).$\end{enumerate}}
\begin{proof}    To prove the first part, for $ 1 \leq i \leq \ell,$ let $\pi_i$ be the projection of $V$ onto the ring $V_i.$ Then for each $i,$ we see that $\pi_i(\mathcal{C})$ is a $\lambda_i$-constacyclic code of length $m_i$ over $\mathbb{F}_{q}$ having the generator polynomial as $w_i(x)=\text{gcd}(a_{1,i}(x),a_{2,i}(x),\cdots,a_{\rho,i}(x), x^{m_i} - \lambda_i).$ From this, we observe that $\underset{1 \le i \le \ell}{\text{lcm}} \left[ \frac{x^{m_i}-\lambda_i}{w_i(x)} \right]$ is an annihilating polynomial of the code $\mathcal{C},$ so $h(x)$ divides $\underset{1 \le i \le \ell}{\text{lcm}}  \left[\frac{x^{m_i}-\lambda_i}{w_i(x)}\right].$ On the other hand,  since $h(x)$ is the parity-check polynomial of $\mathcal{C},$ we must have  $a_{\kappa,i}(x)h(x)=0$ in the ring $\frac{\Fq [x]}{\left< x^{m_i} - \lambda_i \right>}$ for  $1 \leq \kappa \leq \rho$ and $1 \leq i \leq \ell.$ This implies that $x^{m_i}-\lambda_i$ divides $h(x) \text{gcd}(a_{1,i}(x), a_{2,i}(x), \cdots, a_{\rho, i}(x))$ in $\Fq[x],$ which further implies that $\frac{x^{m_i} - \lambda_i}{w_i(x)}$ divides $h(x)$ for each $i.$ This shows that $\underset{1 \le i \le \ell}{\text{lcm}}  \left[ \frac{x^{m_i} - \lambda_i}{w_i(x)}\right]$ divides $h(x)$ in $\Fq[x].$ From this, we get $h(x)= \underset{1 \le i \le \ell}{\text{lcm}} \left[ \frac{ {x^{m_i} - \lambda_i}}{{w_i(x)}}\right].$

To prove the second part, let $\rho=1$ so that $\mathcal{C}=\langle a_1(x) \rangle.$ Now define a map $\Xi: \mathbb{F}_{q}[x]\rightarrow V$ as $\Xi(\alpha(x))=\alpha(x)a_1(x)$ for each $\alpha(x) \in \mathbb{F}_{q}[x].$ We see that $\Xi$ is an $\mathbb{F}_{q}[x]$-module homomorphism with kernal as $\langle h(x)\rangle$ and image as $\mathcal{C}.$ From this, we get $\mathbb{F}_{q}[x]/\langle h(x) \rangle \simeq \mathcal{C},$ which implies that $\text{dim}_{\mathbb{F}_{q}}\mathcal{C}=\text{deg }h(x).$
\end{proof}
However, in the following example, we observe that Theorem \ref{parity0}(b) does not hold for a $\rho$-generator $\Lambda$-multi-twisted code with $\rho \geq 2.$

{\ex Let $q=2,$ $\ell=3,$ $m_1=3,$ $m_2=5,$ $m_3=7$ and $\lambda_1=\lambda_2=\lambda_3=1,$ so that $\Lambda=(1,1,1).$ Let $\mathcal{C}$ be a $2$-generator $\Lambda$-multi-twisted code  length $15$ over $\mathbb{F}_2,$ whose generating set is $\{(x^2+1,x^3+x,x^3+x+1),(x^2+x,x^4+x^3+x^2+x+1,x^3+x^2+1)\}.$ Here $V= V_1 \times V_2 \times V_3,$ where $V_1=\frac{\mathbb{F}_2[x]}{\left< x^3 -1 \right>} ,$ $V_2= \frac{\mathbb{F}_2[x]}{\left< x^5 -1 \right>} $ and $V_3=\frac{\mathbb{F}_2[x]}{\left< x^7-1 \right>}.$ In order to write down the decomposition of $V,$ we see that $x^3-1=(x+1)(x^2+x+1),$ $ x^5-1=(x+1)(x^4+x^3+x^2+x+1)$ and $ x^7-1=(x+1)(x^3+x^2+1)(x^3+x+1)$ are irreducible factorizations of $x^3-1,$ $x^5-1$ and $x^7-1$ over $\mathbb{F}_2.$ Let us take $g_1(x)=x-1,$ $g_2(x)=x^2+x+1,$ $g_3(x)=x^4+x^3+x^2+x+1,$ $g_4(x)=x^3+x^2+1$ and $g_5(x)=x^3+x+1,$ so that $F_w = \mathbb{F}_{2}[x]/\langle g_w(x)\rangle$ for $1 \leq w \leq 5.$ Note that $F_1 \simeq \mathbb{F}_{2},$ $F_2 \simeq \mathbb{F}_{4},$ $F_3 \simeq \mathbb{F}_{16}$ and $F_4 \simeq F_5 \simeq \mathbb{F}_{8}.$ By applying Chinese remainder Theorem, we get $V=\big( F_1,F_1,F_1\big) \oplus(F_2,0,0\big) \oplus \big(0,F_3 ,0\big) \oplus \big(0,0,F_4\big) \oplus \big(0,0,F_5\big).$ From this and applying Theorem \ref{t11}, we see that the constituents of $\mathcal{C}$ are given by  $\mathcal{C}_1 = \left< (0,0,1),(0,1,1)\right>,$ $\mathcal{C}_2 = \left< (\alpha_2,0,0),(1,0,0)\right>$  with $\alpha_2^2+\alpha_2+1=0,$ $ \mathcal{C}_3 = \left< (0,\alpha_3^3+\alpha_3,0),(0,0,0)\right>$  with $\alpha_3^4+\alpha_3^3 +\alpha_3^2 + \alpha_3 +1=0,$  $\mathcal{C}_4 = \left< (0,0, \alpha_4 + \alpha_4^2),(0,0,0)\right>$  with $\alpha_4^3 + \alpha_4^2 +1=0$ and $\mathcal{C}_5 = \left< (0,0,0),(0,0, \alpha_5 + \alpha_5^2)\right>$  with $\alpha_5^3 + \alpha_5 +1=0.$ We observe that $\text{dim}_{F_1}\mathcal{C}_{1}=2$ and $\text{dim}_{F_2}\mathcal{C}_{2}=\text{dim}_{F_3}\mathcal{C}_{3}=\text{dim}_{F_4}\mathcal{C}_{4}=\text{dim}_{F_5}\mathcal{C}_{5}=1.$ Using this and by applying Theorem \ref{t11} again, we get  $\text{dim}_{\mathbb{F}_2}\mathcal{C} =\sum\limits_{w=1}^{5}\text{dim}_{F_w}\mathcal{C}_{w}~ \text{deg }g_w(x)= 14.$ On the other hand, by applying Theorem \ref{parity0}(a), we get $h(x)= (x+1)(x^4+x^3+x^2+x+1)(x^3+x^2+1)(x^3+x+1)(x^2+x+1),$ which implies that $\text{deg }h(x)=13.$ This shows that $\text{dim}_{\mathbb{F}_2}\mathcal{C} \neq \text{deg }h(x)$ in this case.}

In the following theorem, we determine generating sets of dual codes of some $\rho$-generator multi-twisted codes of length $n$ over $\mathbb{F}_{q}.$
{\thm\label{parity1} Let $x^{m_1}-\lambda_1, x^{m_2}-\lambda_2,\cdots,x^{m_{\ell}}-\lambda_{\ell}$ be pairwise coprime polynomials in $\Fq[x].$ Let $\mathcal{C}=\big< {a_1}(x),{a_2}(x), \\ \cdots,{a_{\rho}}(x) \big>$ be a $\rho$-generator $\Lambda$-multi-twisted  code of length $n$ over $\mathbb{F}_{q},$ where ${a_\kappa}(x)=(a_{\kappa,1}(x),a_{\kappa,2}(x),\cdots,a_{\kappa,\ell}(x))$ for  $1 \le \kappa \le \rho.$ Then we have $$\mathcal{C}^{\perp}=\left< {H_1}(x), {H_2}(x), \cdots, {H_{\ell}}(x) \right>,$$ where ${H_i}(x) = (0,\cdots,0,\overline{h_i(x)},0,\cdots,0)$ with $h_i(x)=(x^{m_i} - \lambda_i)/{gcd(a_{1,i}(x),a_{2,i}(x),\cdots,a_{\rho,i}(x), x^{m_i} - \lambda_i)}$ for $1 \le i \le \ell.$ (Here $\overline{h_i(x)}=h_i(x^{-1})$ with $x^{-1}= \lambda_i^{-1}x^{m_i-1}$ for each $i.$)}
\begin{proof} In order to prove this, we see that $\left({a_\kappa}(x),{H_i}(x)\right)= {a_{\kappa,i}(x)h_i(x)\lambda_i  (x^m -1)}/{(x^{m_i}- \lambda_i)}={a_{\kappa,i}(x)\lambda_i (x^m -1)}/$ ${\text{gcd}(a_{1,i}(x),a_{2,i}(x),\cdots,a_{\rho,i}(x), x^{m_i} - \lambda_i)}=0$ in $\frac{\mathbb{F}_{q}[x]}{\langle x^m-1 \rangle}$ for $1 \le \kappa \le \rho$ and $1 \le i \le \ell.$ This implies that ${H_i}(x) \in \mathcal{C}^{\perp}$ for  each $i.$ Now let ${b}(x)=(b_1(x),b_2(x), \cdots, b_{\ell}(x)) \in \mathcal{C}^{\perp}.$ Then we have $\left({a_\kappa}(x),{b}(x)\right)=0$ in $\frac{\mathbb{F}_{q}[x]}{\langle x^m-1 \rangle}$ for $1 \le \kappa \le \rho.$ From this, we see that $x^m-1$ divides $\sum\limits_{i=1}^{\ell} a_{\kappa,i}(x)\overline{b_i(x)}\lambda_i(x^m-1)/(x^{m_i} -\lambda_i),$ which implies that $x^{m_j}-\lambda_j$ divides $\sum\limits_{i=1}^{\ell} a_{\kappa,i}(x)\overline{b_i(x)}\lambda_i(x^m-1)/(x^{m_i} -\lambda_i)$ for $1 \leq \kappa \leq \rho$ and $1 \leq j \leq \ell.$ As $(x^{m_i}-\lambda_i, x^{m_j}-\lambda_j)=1$ for all $j \neq i,$  $x^{m_j} - \lambda_j$ divides $a_{\kappa,j}(x)\overline{b_j(x)}$ for each $\kappa.$ This implies that $\overline{h_j(x)}$ divides $b_j(x)$ for each $j.$ This gives   ${b}(x) \in \left< {H_1}(x), {H_2}(x), \cdots, {H_{\ell}}(x) \right>,$ from which the desired result follows.
\end{proof}
In the following theorem, we obtain another lower bound, {\it viz.} the BCH type lower bound, on the minimum Hamming distances of $\rho$-generator $\Lambda$-multi-twisted codes of length $n$ over $\mathbb{F}_{q}.$
{\thm\label{parity2} Let $\mathcal{C}=\big< {a_1}(x),{a_2}(x),\cdots,{a_{\rho}}(x) \big>$ be a $\rho$-generator $\Lambda$-multi-twisted  code of  length $n$ over $\mathbb{F}_{q},$ where ${a_\kappa}(x)=(a_{\kappa,1}(x),a_{\kappa,2}(x),\cdots,a_{\kappa,\ell}(x))$ for  $1 \le \kappa \le \rho.$ Then the minimum Hamming distance $d_{\min}(\mathcal{C})$ of the code $\mathcal{C}$ satisfies \vspace{-1mm}$$d_{\min}(\mathcal{C}) \geq \underset{1 \leq i \leq \ell}{\text{min}}(b_i+1),\vspace{-1mm}$$ where for each $i,$ $b_i$ is the maximum number of consecutive exponents of zeros of $\text{gcd}(a_{1,i}(x), a_{2,i}(x),\cdots, a_{\rho,i}(x), x^{m_i}- \lambda_i )$ over  $\Fq.$}
\begin{proof} To prove this, let ${B_i}(x)=(0,\cdots,0, \underbrace{w_i(x)}_{i^{th}},0,\cdots,0) \in \mathscr{R},$ where $w_i(x)=\gcd(a_{1,i}(x), a_{2,i}(x),\cdots, a_{\rho,i}(x),$ $ x^{m_i}- \lambda_i)$ for $1 \leq i \leq \ell.$ Now let $\mathcal{C}'=\left<{B}_1(x),{B}_{2}(x),\cdots,{B}_{\ell}(x)\right>$ be a $\Lambda$-multi-twisted code of length $n$ over $\mathbb{F}_{q}.$ 
 Here for $1 \leq \kappa \leq \rho,$ we observe that ${a_\kappa}(x)=\sum\limits_{i=1}^{\ell} \frac{a_{\kappa,i}(x)}{w_i(x)}B_i(x),$ which implies that $\mathcal{C} \subseteq \mathcal{C}'.$ From this, we obtain $d(\mathcal{C}) \geq d(\mathcal{C}').$ Next for $1 \leq i \leq \ell,$ if $\pi_i$ is the projection of $V$ onto $\frac{\mathbb{F}_{q}[x]}{\left<x^{m_i}-\lambda_i\right>},$ then $\pi_i(\mathcal{C}')$ is a $\lambda_i$-constacyclic code of length $m_i$ over $\mathbb{F}_{q}$ having the generator polynomial as $w_i(x).$ Now if $b_i$ is the maximum number of consecutive exponents of zeros of $w_i(x),$ then working in a similar manner as in Theorem 8 of  \cite[Ch. 7]{mac}, we see  that $d(\pi_i(\mathcal{C}')) \geq b_i +1.$ Further, we observe that if the $i$th block $c_i \in \mathbb{F}_{q}^{m_i}$ of a codeword  $c=(c_1,c_2,\cdots,c_{\ell}) \in \mathcal{C}'$ is non-zero, then the Hamming weight $w_H(c_i)$ of $c_i$ satisfies $w_H(c_i) \geq b_i+1.$ This implies that $w_H(c)\geq\min\limits_{1 \leq i \leq \ell}(b_i+1)$ for each $c(\neq 0) \in \mathcal{C}'.$ From this, we obtain the desired result.
\end{proof}
Next for $\Lambda =(\lambda_1,\lambda_2,\cdots,\lambda_{\ell}),$   $\Omega=(\omega_1,\omega_2,\cdots,\omega_{\ell}) \in \mathbb{F}_{q}^{\ell},$ let us define $\mathcal{I}_{\Lambda,\Omega}= \{i: 1\leq i\leq \ell ,\lambda_i \neq \omega_i\}$ and  $\Lambda-\Omega = (\lambda_1-\omega_1,\lambda_2-\omega_2,\cdots, \lambda_{\ell}-\omega_{\ell}).$ For $1 \leq i \leq \ell,$ let $\pi_i$ be the projection of $V$ onto $V_i.$ If $\mathcal{C}$ is a $\Lambda$-multi-twisted code of length $n$ over $\mathbb{F}_{q},$ then  one can easily observe that $\pi_i(\mathcal{C})$ is a $\lambda_i$-constacyclic code of length $m_i$ over $\mathbb{F}_{q}$ for $1 \leq i \leq \ell.$ In the following theorem, we obtain a lower bound on the dimension of some $[\Lambda,\Omega]$-multi-twisted codes of length $n$ over $\mathbb{F}_{q},$ where $\Lambda \neq \Omega.$ It extends Theorem 1 of Saleh and Esmaeili \cite{saleh}.
\begin{thm}\label{t23}  
Let $\Lambda =(\lambda_1,\lambda_2,\cdots,\lambda_{\ell})$ and $ \Omega=(\omega_1,\omega_2,\cdots,\omega_{\ell}),$  where $\lambda_i, \omega_i$'s are non-zero elements of $\mathbb{F}_{q}.$  Let $\mathcal{C}$ be a $\Lambda$-multi-twisted  and  $\Omega$-multi-twisted code of length $n$ over $\mathbb{F}_{q}.$ Let $\mathcal{J}_{\mathcal{C}}=\{i: 1 \leq i \leq \ell,\pi_i(\mathcal{C}) \neq \{0\}\}.$   If $\mathcal{I}_{\Lambda,\Omega}  \cap \mathcal{J}_{\mathcal{C}}$ is non-empty,  then  $\text{dim}_{\mathbb{F}_{q}}\mathcal{C} \geq \max\limits_{i \in \mathcal{I}_{\Lambda,\Omega}  \cap \mathcal{J}_{\mathcal{C}}}\{m_i\}.$ As a consequence, if $\lambda_i \neq \omega_i$ and $\pi_i(\mathcal{C}) \neq \{0\}$ for $1 \leq i \leq \ell, $ then we have $\text{dim}_{\mathbb{F}_{q}}\mathcal{C} \geq \max\{m_1,m_2,\cdots, m_{\ell}\}.$\end{thm}
\begin{proof} For each $i \in \mathcal{I}_{\Lambda,\Omega}  \cap \mathcal{J}_{\mathcal{C}},$ as $\pi_i(\mathcal{C}) \neq \{0\},$ there exists a codeword  $c=(c_{1,0},c_{1,1},\cdots,c_{1,m_1-1};c_{2,0},c_{2,1},\cdots,c_{2,m_2-1};\\\cdots;c_{\ell,0},c_{\ell,1},\cdots,c_{\ell,m_{\ell}-1})\in \mathcal{C}$ such that $c_{i,m_i-1} \neq 0.$ As $\mathcal{C}$ is a both $\Lambda$-multi-twisted and  $\Omega$-multi-twisted code, we note that $T_{\Lambda}(c),T_{\Omega}(c) \in \mathcal{C},$ which implies that $T_{\Lambda-\Omega}(c)=T_{\Lambda}(c)-T_{\Omega}(c)=((\lambda_1-\omega_1)c_{1,m_1-1},0,\cdots,0;(\lambda_2-\omega_2)c_{2,m_2-1},0,\cdots,0;\cdots;(\lambda_{\ell}-\omega_{\ell})c_{\ell,m_{\ell}-1},0,\cdots,0)\in \mathcal{C}.$ Further, for each $i \in \mathcal{I}_{\Lambda,\Omega}  \cap \mathcal{J}_{\mathcal{C}},$  we see that $(\lambda_i-\omega_i)c_{i,m_i-1} $ is non-zero, which implies that $ T_{\Lambda-\Omega}(c), T_{\Lambda-\Omega}^{2}(c), \cdots, T_{\Lambda-\Omega}^{m_i}(c) \in \mathcal{C}$ are linearly independent over $\mathbb{F}_{q},$ and hence $\text{dim}_{\mathbb{F}_{q}}\mathcal{C} \geq m_i.$ From this, it follows that $\text{dim}_{\mathbb{F}_{q}}\mathcal{C} \geq \max\limits_{i \in \mathcal{I}_{\Lambda,\Omega}  \cap \mathcal{J}_{\mathcal{C}}}\{m_i\}.$
 \end{proof}
 A $\Lambda$-multi-twisted code $\mathcal{C}$ of length $n$ over $\mathbb{F}_{q}$ is said to be LCD if it satisfies $\mathcal{C} \cap \mathcal{C}^{\perp}=\{0\}.$ In the following two theorems, we derive sufficient conditions under which a $\Lambda$-multi-twisted code is LCD  extending Theorems 2 and 3 of Saleh and Esmaeili \cite{saleh}. However, these conditions are not necessary for a $\Lambda$-multi-twisted code to be LCD, which we shall illustrate in Examples \ref{ex} and \ref{exx}. 
  \begin{thm} \label{L1}Let $\Lambda =(\lambda_1,\lambda_2,\cdots,\lambda_{\ell}),$  where $\lambda_1,\lambda_2,\cdots,\lambda_{\ell}$ are non-zero elements of $\mathbb{F}_{q}$ satisfying $\lambda_i\neq \lambda_i^{-1}$ for $1 \leq i \leq {\ell}.$ Let $\mathcal{C}$ be a $\Lambda$-multi-twisted code of length $n$ over $\mathbb{F}_{q}.$ Then the following hold.
 \begin{enumerate} \item[(a)] If either $\text{dim}_{\mathbb{F}_{q}}\mathcal{C} < \min\limits_{1 \leq i \leq {\ell} }\{m_i\}$ or $\text{dim}_{\mathbb{F}_{q}}\mathcal{C}^\perp  < \min\limits_{1 \leq i \leq {\ell}  }\{m_i\},$ then $\mathcal{C}$ is an LCD  code.\item[(b)]  If  $\text{dim}_{\mathbb{F}_{q}}\mathcal{C} =\min\limits_{1 \leq i \leq {\ell}  }\{m_i\},$ then $\mathcal{C}$ is either an LCD or a self-orthogonal code. \item[(c)] If  $\text{dim}_{\mathbb{F}_{q}}\mathcal{C}^{\perp}=\min\limits_{1 \leq i \leq {\ell}  }\{m_i\},$ then $\mathcal{C}$ is either an LCD or  a dual-containing code, i.e., $\mathcal{C}^{\perp}\subseteq \mathcal{C}.$ \item[(d)] If $\text{dim}_{\mathbb{F}_{q}}\mathcal{C} =\text{dim}_{\mathbb{F}_{q}}\mathcal{C}^{\perp}=\min\limits_{1 \leq i \leq {\ell}  }\{m_i\},$ then $\mathcal{C}$ is either an LCD or a self-dual code.
 \end{enumerate}\end{thm}
\begin{proof} \begin{enumerate}\item[(a)] Note that $\mathcal{C} \cap \mathcal{C}^{\perp}$ is both a $\Lambda$-multi-twisted and a $\Lambda'$-multi-twisted code of length $n$ over $\mathbb{F}_{q}.$ We assert that $\mathcal{C} \cap \mathcal{C}^{\perp}=\{0\}.$ Then by Theorem \ref{t23}, we get $\text{dim}_{\mathbb{F}_{q}} (\mathcal{C} \cap \mathcal{C}^{\perp} )\geq \min\limits_{1 \leq i \leq {\ell}  }\{m_i\}.$ Since $\mathcal{C}\cap \mathcal{C}^{\perp}$ is a subspace of both $\mathcal{C}$ and $\mathcal{C}^{\perp},$ we get $\text{dim}_{\mathbb{F}_{q}}\mathcal{C} \geq \min\limits_{1 \leq i \leq {\ell} }\{m_i\}$ and $\text{dim}_{\mathbb{F}_{q}}\mathcal{C}^{\perp} \geq \min\limits_{1 \leq i \leq {\ell} }\{m_i\},$ which contradicts our hypothesis. So we must have $\mathcal{C} \cap \mathcal{C}^{\perp}=\{0\}.$  \item[(b)] If $\mathcal{C} \cap \mathcal{C}^{\perp} \neq \{0\},$ then working as in part (a), we see that $\text{dim}_{\mathbb{F}_{q}} (\mathcal{C} \cap \mathcal{C}^{\perp} )\geq \min\limits_{1 \leq i \leq {\ell} }\{m_i\}.$ Now as $\text{dim}_{\mathbb{F}_{q}}\mathcal{C} =\min\limits_{1 \leq i \leq {\ell} }\{m_i\},$ we get $\mathcal{C} \cap \mathcal{C}^{\perp}=\mathcal{C},$ which implies that $\mathcal{C} \subseteq \mathcal{C}^{\perp}.$ This proves (b).\item[(c)] Its proof is similar to that of part (b).
\item[(d)] It follows immediately from parts (b) and (c).\end{enumerate}\vspace{-5mm}\end{proof}
\begin{thm} \label{L2} Let $\Lambda=(\lambda_1,\lambda_2,\cdots,\lambda_{\ell}),$ where $\lambda_1,\lambda_2,\cdots,\lambda_{\ell}$ are non-zero elements of $\mathbb{F}_{q}$ satisfying $\lambda_i \neq \lambda_i^{-1}$ for $1 \leq i \leq \ell.$  Let $\mathcal{C}$  be a $\rho$-generator $\Lambda$-multi-twisted  code of length $n$ over $\mathbb{F}_{q}$ such that either $\pi_i(\mathcal{C}) \neq <1>$ or $\pi_i(\mathcal{C}^{\perp}) \neq  <1>$ for $1 \leq i \leq \ell.$  Then $\mathcal{C}$ is an LCD code.
\end{thm}
\begin{proof} For  $1 \leq i \leq \ell,$ we see that the linear code $\pi_i(\mathcal{C})\cap \pi_i(\mathcal{C^\perp})$ is both $\lambda_i$-constacyclic and $\lambda_i^{-1}$-constacyclic code of length $m_i$ over $\mathbb{F}_{q}.$ Further, for each $i,$ as $\lambda_i \neq \lambda_i^{-1},$ by Corollary 2.7 of Dinh \cite{Dinh}, we see that either $\pi_i(\mathcal{C})\cap \pi(\mathcal{C}^\perp)= \{0\}$ or $\pi_i(\mathcal{C})\cap \pi(\mathcal{C}^\perp)= \langle 1 \rangle.$  Now since either $\pi_i(\mathcal{C}) \neq \langle 1 \rangle $ or $\pi_i(\mathcal{C}^\perp) \neq \langle 1 \rangle,$ we get $\pi_i(\mathcal{C})\cap \pi(\mathcal{C}^\perp)= \{0\}$ for each $i.$ As $\pi_i(\mathcal{C} \cap \mathcal{C}^{\perp})$ is a subspace of $\pi_i(\mathcal{C}) \cap \pi_i(\mathcal{C}^{\perp}),$ we get $\pi_i(\mathcal{C}\cap \mathcal{C}^{\perp})={0}$ for $1 \leq i \leq \ell.$ This implies that $\mathcal{C} \cap \mathcal{C}^{\perp}=\{0\},$ i.e.,  $\mathcal{C}$ is an LCD code.
\end{proof}

{\cor Let $\lambda_1,\lambda_2,\cdots,\lambda_{\ell} \in \mathbb{F}_{q}\setminus\{0\}$ be such that $\lambda_i \neq \lambda_i ^{-1}$ for all $1 \leq i \leq {\ell}$ and the polynomials $x^{m_1}-\lambda_1, x^{m_2}-\lambda_2,\cdots,x^{m_{\ell}}-\lambda_{\ell}$ are pairwise coprime in $\Fq[x].$ Then any $(\lambda_1,\lambda_2,\cdots,\lambda_{\ell})$-multi-twisted code of length $n$ over $\mathbb{F}_{q}$ is an LCD code.  }
\begin{proof} It follows immediately from Theorems \ref{parity1} and \ref{L2}.
\end{proof}

In the following example, we  illustrate Theorems \ref{L1}(a) and \ref{L2}. 
{\ex Let $ q=5,$ $\ell=2,$ $m_1=m_2=3,$  $\Lambda =(3,2)$  and $\mathbb{F}_{5}=\mathbb{Z}_{5}.$ Here we have $ V=V_1 \times V_2=\frac{\mathbb{F}_{5}[x]}{\langle x^3-3 \rangle} \times \frac{\mathbb{F}_{5}[x]}{\langle x^3-2 \rangle}.$ Now we see that the irreducible factorizations of the polynomials $x^3-3$ and $x^3-2$ over $\mathbb{F}_{5}$ are given by $x^3-3 = (x+3)(x^2+2x+4)$ and $x^3-2 = (x+2)(x^2+3x+4) ,$ respectively. Let $\mathcal{C}$ be a $1$-generator $\Lambda$-multi-twisted code of length $6$ over $\mathbb{F}_5$ with the  generating set as $\{(x+3,x+2)\}.$ It is easy to observe that $\pi_1(\mathcal{C})=\big<\gcd(x+3,x^3-3)\big> =\big<x+3\big> \neq \big< 1 \big>$ and $ \pi_2(\mathcal{C})=\big<\gcd(x+2,x^3-2)\big> = \big<x+2\big> \neq \big< 1 \big>.$ So by Theorem \ref{L2}, we see that $\mathcal{C}$ is an LCD code.  On the other hand, we note that $ V'=V_1' \times V_2'=\frac{\mathbb{F}_{5}[x]}{\langle x^3-2 \rangle} \times \frac{\mathbb{F}_{5}[x]}{\langle x^3-3 \rangle}.$  By Theorem \ref{parity1}, we obtain $\mathcal{C}^\perp = \big< (x^2+3x+4,0),(0,x^2+2x+4)\big>. $ It is easy to see that $\mathcal{C}_{1}^{\perp}=\text{Span}_{F_1} \{ (2,0)\},$ $\mathcal{C}_{2}^{\perp}=\text{Span}_{F_2}\{(0,2)\}$ and $\mathcal{C}_{3}^{\perp}=\mathcal{C}_{4}^{\perp}=\{0\},$ where $F_1 \simeq F_2 \simeq \mathbb{F}_{5}.$ Using Theorem \ref{t11}, we get $\text{dim}_{\mathbb{F}_{5}}\mathcal{C}^{\perp}=2.$ By applying Theorem \ref{L1}(a) also, we see that $\mathcal{C}$ is an LCD code.}

In the following example, we  show that  the sufficient conditions derived in Theorems \ref{L1}(a) are not necessary for a $\Lambda$-multi-twisted code to be LCD. 
 {\ex \label{ex} Let $ q=7,$ $\ell=2,$ $m_1=m_2=2,$  $\Lambda =(2,5)$  and $\mathbb{F}_{7}=\mathbb{Z}_{7}.$ Here we have $ V=V_1 \times V_2=\frac{\mathbb{F}_{7}[x]}{\langle x^2-2 \rangle} \times \frac{\mathbb{F}_{7}[x]}{\langle x^2-5 \rangle}$ and  $ V'=V_1' \times V_2'=\frac{\mathbb{F}_{7}[x]}{\langle x^2-4 \rangle} \times \frac{\mathbb{F}_{7}[x]}{\langle x^2-3 \rangle}.$ It is easy to see that the polynomials $x^2-3$ and $x^2-5$ are irreducible over $\mathbb{F}_{7},$ and that the irreducible factorizations of the polynomials $x^2-3$ and $x^2-5$ over $\mathbb{F}_{7}$ are given by $x^2-2 = (x+3)(x+4)$ and $x^2-4 = (x+2)(x+5) ,$ respectively. Let $\mathcal{C}$ be a $1$-generator $\Lambda$-multi-twisted code of length $4$ over $\mathbb{F}_7$ with the  generating set as $\{(x+1,0)\}.$ It is easy to observe that $\pi_1(\mathcal{C})=\big<\gcd(x+1,x^2-2)\big> =\big<1\big> $ and $ \pi_2(\mathcal{C})=\big<\gcd(0,x^2-5)\big> = \{0\}.$ Further, as the polynomials $x^2-2$ and $x^2-5$ are coprime over $\mathbb{F}_{7},$ by applying Theorem \ref{parity1}, we obtain $\mathcal{C}^\perp = \big< (0,0),(0,1)\big> .$ From this, we get $\pi_1(\mathcal{C}^\perp)=\big<\gcd(0,0,x^2-4)\big> =\{0\}$ and $ \pi_2(\mathcal{C}^\perp)=\big<\gcd(0,1,x^2-3)\big> = \big<1\big> .$ Therefore by Theorem \ref{L2}, we see that $\mathcal{C}$ is an LCD code.  On the other hand, by applying Theorem \ref{parity0}(a), we get $h(x) = x^2-2.$ Using Theorem \ref{parity0}(b), we get $\text{dim}_{\mathbb{F}_{7}}\mathcal{C}=2.$ Further, it is easy to see that  $\mathcal{C}_{1}^{\perp}=\mathcal{C}_{2}^{\perp}=\{0\}$ and $\mathcal{C}_{3}^{\perp}=\text{Span}_{F_3}\{(0,1)\},$ where $F_3=\frac{\mathbb{F}_{7}[x]}{\langle x^2-3 \rangle}\simeq \mathbb{F}_{49}.$ By  Theorem \ref{t11}, we get $\text{dim}_{\mathbb{F}_{7}}\mathcal{C}^{\perp}=2.$ This shows that the code $\mathcal{C}$ does not satisfy hypotheses of  Theorem \ref{L1}(a).  }
 
 In the following example, we  show that  the sufficient conditions derived in Theorems \ref{L2} are not necessary for a $\Lambda$-multi-twisted code to be LCD. 

   {\ex \label{exx} Let $ q=5,$ $\ell=2,$ $m_1=m_2=3,$  $\Lambda =(3,3)$  and $\mathbb{F}_{5}=\mathbb{Z}_{5}.$ Here we have $ V=V_1 \times V_2=\frac{\mathbb{F}_{5}[x]}{\langle x^3-3 \rangle} \times \frac{\mathbb{F}_{5}[x]}{\langle x^3-3 \rangle}$ and  $ V'=V_1' \times V_2'=\frac{\mathbb{F}_{5}[x]}{\langle x^3-2 \rangle} \times \frac{\mathbb{F}_{5}[x]}{\langle x^3-2 \rangle}.$ It is easy to see that the irreducible factorizations of the polynomials $x^3-3$ and $x^3-2$ over $\mathbb{F}_{5}$ are given by $x^3-3 = (x-2)(x^2+2x+4)$ and $x^3-2 = (x-3)(x^2+3x+4) ,$ respectively. Now let $g_1(x)=x-2,g_2(x)=x^2+2x+4,h_1(x)=x-3 \text{ and } h_2(x)=x^2+3x+4.$ Here we can easily observe that $g_1^*(x)= h_1(x)$ and $g_2^*(x)= h_2(x).$ Let $\mathcal{C}$ be a $1$-generator $\Lambda$-multi-twisted code of length $6$ over $\mathbb{F}_{5}$ with the  generating set as $\{(1,x+1)\}.$  By applying Chinese Remainder Theorem, we get $V=\big( F_1 ,F_1 \big) \oplus(F_2 , F_2 \big)$ and $V'=\big( H_1, H_1 \big) \oplus(H_2,H_2 \big),$ where $F_w = \frac{\mathbb{F}_{5}[x]}{\langle g_w(x) \rangle}$ and $H_w = \frac{\mathbb{F}_{5}[x]}{\langle h_w(x) \rangle}$ for $ 1 \leq w \leq 2.$ From this and applying Theorem \ref{t11}, we see that the constituents of $\mathcal{C}$ are given by  $\mathcal{C}_1 = \left< (1,3)\right>$ and $\mathcal{C}_2 = \left< (1,x+1)\right>.$ Further, in view of  \eqref{dbform}, we obtain $\mathcal{C}_1^\perp = \left< (-3,1)\right>$ and $\mathcal{C}_2^\perp = \left< (1,2x+3)\right>.$ Now  by applying Chinese Remainder Theorem, we get $ \mathcal{C}^\perp$ is generated by $(-2x^2-x+3,x^2+2).$ Moreover, it is easy to see that $\pi_1(\mathcal{C})=\big<\gcd(1,x^3-3)\big> =\big<1\big> ,$ $ \pi_2(\mathcal{C})=\big<\gcd(x+1,x^3-3)\big> = \big<1\big>,$ $\pi_1(\mathcal{C}^\perp)=\big<\gcd(-2x^2-x+3,x^3-2)\big> =\big<1\big>$ and $\pi_2(\mathcal{C}^\perp)=\big<\gcd(x^2+2,x^3-2)\big> =\big<1\big>,$ which shows that the code $\mathcal{C}$ does not satisfy the hypotheses of  Theorem \ref{L2}. On the other hand,  by  Theorem \ref{t11}, we have $\text{dim}_{\mathbb{F}_{5}}\mathcal{C}=3$ and $\text{dim}_{\mathbb{F}_{5}}\mathcal{C}^\perp =3.$ It is easy to observe that $\mathcal{C} \neq \mathcal{C}^{\perp}.$ Therefore by Theorem \ref{L1}(d), we see that $\mathcal{C}$ is  an LCD  code. }

\section{Trace description of $\Lambda$-multi-twisted codes}\label{Traceformula}
In this section, we shall provide a trace description for $\Lambda$-multi-twisted codes of length $n$ over $\mathbb{F}_{q}$ by extending the work of G\"{u}neri et al. \cite{cem} to $\Lambda$-multi-twisted codes.  
Towards this, for $1 \leq w \leq r$ and $1 \leq i \leq {\ell},$ we recall that if $\epsilon_{w,i}=1,$ then $g_w(x)$ divides $x^{m_i}-\lambda_i$ in $\mathbb{F}_{q}[x],$ and the ideal  $\big\langle \frac{x^{m_i}-\lambda_i}{g_w(x)}\big\rangle$ is a minimal $\lambda_i$-constacyclic code  of length $m_i$ over $\mathbb{F}_{q},$ whose generating idempotent is denoted by  $\Theta_{w,i}.$ If $\epsilon_{w,i}=0$ for some $w$ and $i,$ then  we shall denote the zero codeword of length $m_i$  by $\Theta_{w,i}.$  Now by Theorem 3.1 of Sharma and Rani \cite{srani}, we see that there exist ring isomorphisms $\phi_{w,i}:\langle \Theta_{w,i} \rangle \rightarrow \epsilon_{w,i}F_w$ and $\psi_{w,i}: \epsilon_{w,i}F_w \rightarrow \langle \Theta_{w,i} \rangle,$ defined as  $\phi_{w,i}(a(x))=\epsilon_{w,i}a(\alpha_w)\text{ for each }a(x) \in \langle \Theta_{w,i} \rangle$  and 
 \begin{equation}\label{eqnn} 
 \psi_{w,i}(\gamma)=\frac{1}{m_i}\big(Tr_{F_w/\mathbb{F}_{q}}(\gamma),Tr_{F_w/\mathbb{F}_{q}}(\gamma\alpha_w^{-1}), \cdots,Tr_{F_w/\mathbb{F}_{q}}(\gamma\alpha_w^{-(m_i-1)} )\big)\text{ for each }\gamma \in \epsilon_{w,i}F_w, \end{equation} where  $Tr_{F_w/\mathbb{F}_{q}}$ is the trace map from $F_w$ onto $\mathbb{F}_{q} $ and $\alpha_w$ is a zero of $g_w(x)$ in $F_w.$
 Further, note that the ring isomorphisms $\phi_{w,i}$ and $\psi_{w,i}$ are inverses of each other, and that $\psi_{w,i}(\epsilon_{w,i}1_w)=\Theta_{w,i},$ where $1_w$ is the multiplicative identity of $F_w.$ We shall view $V=\prod\limits_{i=1}^{\ell}V_i$ and $\mathcal{G}_{w}= \big(\epsilon_{w,1} F_w, \epsilon_{w,2} F_w,\cdots, \epsilon_{w,\ell} F_w\big)$  as  rings with respect to the coordinate-wise addition $+$ and coordinate-wise multiplication $\odot$ for each $w.$   In view of this, $1_V:=(1,1,\cdots,1)$ and $1_{\mathcal{G}_{w}}:=(\epsilon_{w,1}1_w,\cdots,\epsilon_{w,{\ell}}1_w)$ respectively are the multiplicative identities of $V$ and $\mathcal{G}_{w}$  for each $w.$ Now  for $1 \leq w \leq r,$ let us define the maps  $\Phi_w: V \rightarrow \mathcal{G}_{w}$ and $\Psi_w:\mathcal{G}_{w}\rightarrow V$ as 
$\Phi_w(a_1(x),a_2(x),\cdots,a_{\ell}(x))= (\epsilon_{w,1}a_1(\alpha_w),\epsilon_{w,2}a_2(\alpha_w),\cdots,\epsilon_{w,{\ell}}a_{\ell}(\alpha_w))$ for each $(a_1(x),a_2(x),\cdots,a_{\ell}(x)) \in V$ and $\Psi_w(\gamma_1,\gamma_2,\cdots,\gamma_{\ell})= (\psi_{w,1}(\gamma_1),\\\psi_{w,2}(\gamma_2),\cdots,\psi_{w,{\ell}}(\gamma_{\ell}))$ for each $ (\gamma_1,\gamma_2,\cdots,\gamma_{\ell}) \in \mathcal{G}_{w}.$ Note that both $\Phi_w \text{ and } \Psi_w$ are $\mathbb{F}_{q}$-linear maps and are ring homomorphisms. Moreover, for each $w,$ the restriction map $\Phi_w \restriction_{(\langle \Theta_{w, 1}\rangle, \langle \Theta_{w,2}\rangle, \cdots,\langle \Theta_{w,\ell }\rangle)}$ and the map $\Psi_w$ are inverses of each other. 
 For $1 \leq w \leq r,$ let us define $\Theta_w  = (\Theta_{w,1},\Theta_{w,2},\cdots,\Theta_{w, \ell}).$ It is easy to see that $V=\bigoplus\limits_{w=1}^r\langle\Theta_w\rangle,$ $\sum\limits_{w=1}^{r}\Theta_w=1_V,$  $\langle\Theta_w\rangle = (\langle\Theta_{w,1}\rangle, \langle\Theta_{w,2}\rangle,\cdots,\langle\Theta_{w,\ell}\rangle),$   
 $\Theta_w\odot \Theta_w=\Theta_w,$ $\Theta_{w'}\odot \Theta_w=0$    for each $w\neq w'.$  
 
Next  the concatenation of $\langle \Theta_w \rangle = (\langle \Theta_{w,1}\rangle , \langle \Theta_{w,2}\rangle , \cdots , \langle \Theta_{w,\ell}\rangle) $ and a linear code $\mathcal{D}$ of length $\ell$ over $F_{w} \simeq \mathbb{F}_{q^{d_w}}$ is defined as $\langle\Theta_w\rangle \Box \mathcal{D} = \big\{(\psi_{w,1}(\delta_{w,1}),\psi_{w,2}(\delta_{w,2}),\cdots,\psi_{w,{\ell}}(\delta_{w,{\ell}})):\delta_w=(\delta_{w,1},\delta_{w,2},\cdots,\delta_{w,{\ell}})\in \mathcal{D} \big\}.$ In the following theorem, we shall view $\Lambda$-multi-twisted codes as direct sums of certain concatenated codes. 
{\thm \label{t22} \begin{enumerate}\item[(a)] Let $ \mathcal{C}$ be a $\Lambda$-multi-twisted code of length $n$ over $\mathbb{F}_{q}$ with the constituents as $\mathcal{C}_{1},$ $\mathcal{C}_{2}, \cdots, \mathcal{C}_{r}.$ If $\tilde{\mathcal{C}_w}:=\mathcal{C}\odot \Theta_w  $  for $1 \leq w \leq r,$  then we have $\mathcal{C}=\bigoplus\limits_{w=1}^r \langle\Theta_w\rangle \Box \Phi_w(\tilde{\mathcal{C}_w}).$ Moreover,   $\mathcal{C}_w=\Phi_w(\tilde{\mathcal{C}_w})$ holds for $1 \leq w \leq r.$   As a consequence, we have $\mathcal{C}=\bigoplus\limits_{w=1}^r \langle\Theta_w\rangle \Box \mathcal{C}_w.$
\item[(b)] Conversely, let $\mathfrak{C}_w (\subseteq \mathcal{G}_w )$ be a linear code of length ${\ell}$ over $F_w$ for $1 \leq w \leq r.$ Then $\mathcal{C}=\bigoplus\limits_{w=1}^r \langle\Theta_w\rangle \Box \mathfrak{C}_w$ is a  $\Lambda$-multi-twisted code of  length $n$ over $\mathbb{F}_{q}.$ \end{enumerate}}
\begin{proof} Working in a similar manner as in Theorem 3.4 and Remark 3.5 of G\"{u}neri et al. \cite{cem}, the desired result follows.
 \end{proof}
 In the following theorem, we provide a trace description for $\Lambda$-multi-twisted codes of length $n$ over $\mathbb{F}_{q}$ using their concatenated structure. 
{\thm \label{trace} Let $\mathcal{C}$ be a $\Lambda$-multi-twisted code of length $n$ over $\mathbb{F}_{q}$ with the constituents as $\mathcal{C}_{1}, \mathcal{C}_{2},\cdots, \mathcal{C}_{r}.$  For $\delta_w=(\delta_{w,1},\delta_{w,2},\cdots,\delta_{w,{\ell}}) \in \mathcal{C}_{w}$ with $1 \leq w \leq r,$ let us define \begin{equation*}c_i(\delta_1,\delta_2,\cdots,\delta_{\ell})= \frac{1}{m_i}\Big(\sum\limits_{w=1}^r Tr_{F_w/\mathbb{F}_{q}}(\delta_{w,i}), \sum\limits_{w=1}^rTr_{F_w/\mathbb{F}_{q}}(\delta_{w,i}\alpha_w^{-1}), \cdots, \sum\limits_{w=1}^rTr_{F_w/\mathbb{F}_{q}}(\delta_{w,i}\alpha_w^{-(m_i-1)})\Big) \end{equation*} for $1 \leq i \leq \ell.$  Then we have \begin{equation*}\mathcal{C}=\{ \big(c_1(\delta_1,\delta_2,\cdots,\delta_{\ell}),c_2(\delta_1,\delta_2,\cdots,\delta_{\ell}),\cdots,c_{\ell}(\delta_1,\delta_2,\cdots,\delta_{\ell})\big): \delta_w=(\delta_{w,1},\delta_{w,2},\cdots,\delta_{w,{\ell}}) \in \mathcal{C}_{w} \text{ for }1 \leq w \leq r \} .\end{equation*}}
\vspace{-6mm}\begin{proof} 
By Theorem \ref{t22},  we see that the code $\mathcal{C}$ has the concatenated structure $\mathcal{C}=\bigoplus\limits_{w=1}^r \langle\Theta_w\rangle \Box \mathcal{C}_w,$
where $\langle\Theta_w\rangle \Box \mathcal{C}_w = \big\{(\psi_{w,1}(\delta_{w,1}),\psi_{w,2}(\delta_{w,2}),\cdots,\psi_{w,{\ell}}(\delta_{w,{\ell}})):\delta_w=(\delta_{w,1},\delta_{w,2},\cdots,\delta_{w,{\ell}})\in \mathcal{C}_w \big\}.$  From this, we get 
$\mathcal{C}=\big\{(\sum\limits_{w=1}^r\psi_{w,1}(\delta_{w,1}),\sum\limits_{w=1}^r\psi_{w,2}(\delta_{w,2}),\cdots,\sum\limits_{w=1}^r\psi_{w,{\ell}}(\delta_{w,{\ell}})):\delta_w=(\delta_{w,1},\delta_{w,2},\cdots,\delta_{w,{\ell}})\in \mathcal{C}_w \big\}.$  By \eqref{eqnn}, we see that $\sum\limits_{w=1}^r\psi_{w,i}(\delta_{w,i})= \frac{1}{m_i}\big(\sum\limits_{w=1}^rTr_{F_w/\mathbb{F}_{q}}(\delta_{w,i}), \sum\limits_{w=1}^rTr_{F_w/\mathbb{F}_{q}}(\delta_{w,i}\alpha_w^{-1}), \sum\limits_{w=1}^rTr_{F_w/\mathbb{F}_{q}}(\delta_{w,i}\alpha_w^{-2}),\cdots, \sum\limits_{w=1}^rTr_{F_w/\mathbb{F}_{q}}(\delta_{w,i}\alpha_w^{-(m_i-1)})\big)$ for $1 \leq i \leq \ell,$ from which  the desired result follows.
\end{proof}
We shall illustrate the above theorem in the following example:
{\ex Let $ q=7,$ $\ell=2,$ $m_1=2,$ $m_2=4, $ $\Lambda =(2,4)$ and $\mathbb{F}_{7}=\mathbb{Z}_{7}.$ Here we have $V=V_1 \times V_2,$ where $V_1 =\frac{\mathbb{F}_{7}[x]}{\langle x^2-2 \rangle} $ and $V_2= \frac{\mathbb{F}_{7}[x]}{\langle x^4-4 \rangle} .$  Further, we see that the irreducible factorizations of the polynomials $x^2-2$ and $x^4-4$  over $\mathbb{F}_{7}$ are given by  $x^2-2 = (x+3)(x+4),x^4-4 = (x+3)(x+4)(x^2+2).$ If we take $g_1(x)=x+3,g_2(x)=x+4$ and $g_3(x)=x^2+2,$ then we have $F_1\simeq F_2 \simeq \mathbb{F}_{7}$ and  $F_3\simeq \mathbb{F}_{49}.$ From this and by applying Chinese Remainder Theorem, we get $V \simeq (F_1,F_1) \oplus (F_2,F_2)\oplus (\{0\},F_3).$ Now if $\mathcal{C}$ is a $(2,4)$-multi-twisted code of length $6$ over $\mathbb{F}_{7}$ with the constituents as $\mathcal{C}_{1},$ $\mathcal{C}_{2}$ and $\mathcal{C}_{3},$ then by Theorem \ref{trace}, the code $\mathcal{C}$ is given by
\vspace{2mm}\\ $\left\{\big(\frac{a+c}{2},\frac{2a+5c}{2},\frac{b+d+2e}{4},\frac{2b+5d+2f}{4},\frac{4b+4d+2e\alpha_3^{-2}}{4},\frac{b-d+2f\alpha_3^{-2}}{4}\big):
(a,b)\in \mathcal{C}_1, (c,d) \in \mathcal{C}_2,  (0,e+\alpha_3f) \in \mathcal{C}_3\right\} ,$\vspace{2mm}\\
where  $\alpha_3$ is a root of the polynomial $g_3(x)$ in $F_3.$}

In the following theorem, we obtain a minimum distance bound for $\Lambda$-multi-twisted codes of length $n$ over $\mathbb{F}_{q}$ using their multilevel concatenated structure.
{\thm \label{bound} Let $\mathcal{C}$ be a $\Lambda$-multi-twisted code of length $n$ over $\mathbb{F}_{q}$ with the non-zero constituents as  $\mathcal{C}_{w_1},\mathcal{C}_{w_2},\cdots,\mathcal{C}_{w_t},$ where $1 \leq w_1,w_2,\cdots, w_t \leq r.$  Let $\mathfrak{d}_{j}$ be the minimum Hamming distance of the code $\mathcal{C}_{w_j}$ for $1 \leq j \leq t.$ Let us assume that $\mathfrak{d}_1\leq\mathfrak{d}_2\leq \cdots \leq \mathfrak{d}_t.$ Let us define $\mathfrak{K}_v= \min\limits_{\stackrel{I \subseteq \{1,2,\cdots, \ell\}}{|I|=\mathfrak{d}_v}}\Big\{ \sum\limits_{g\in I} d_{\min}(\langle \theta_{w_1,g} \rangle \oplus \langle \theta_{w_2,g} \rangle \oplus \cdots \oplus \langle \theta_{w_t,g} \rangle)\Big\}$ for $ v \in \{1,2,\cdots,t\}.$ Then the minimum Hamming distance $d_{\min}(\mathcal{C})$ of the code $\mathcal{C}$ satisfies  $$d_{\min}(\mathcal{C})\geq \text{min}\{\mathfrak{K}_1,\mathfrak{K}_2,\cdots,\mathfrak{K}_t\}.$$
(Throughout this paper,  $d_{\min}$ denotes the minimum Hamming distance of the code.)}
\begin{proof}
Working in a similar manner as in Theorem 4.2 of G\"{u}neri et al. \cite{cem}, the desired result follows.
\end{proof}


\section*{Conclusion and future work}
In this paper,  the algebraic structure and duality properties of multi-twisted codes of length $n$ over $\mathbb{F}_{q}$ are studied  with respect to the standard inner product  on $\mathbb{F}_{q}^n.$ A method to construct these codes is provided and two lower bounds on  their minimum Hamming distances are also obtained. It would be interesting to further develop generator theory for these codes and to study their duality properties with respect to  other inner products over finite fields.

\end{document}